\documentclass[oneside,english]{amsart}
\usepackage[T1]{fontenc}
\usepackage[latin9]{inputenc}
\usepackage[hmargin=3cm,vmargin=2.5cm,bindingoffset=0.5cm]{geometry}
\usepackage{amsthm}
\usepackage{amssymb}
\usepackage{graphicx}
\usepackage{esint}

\makeatletter
\numberwithin{equation}{section}
\numberwithin{figure}{section}
\theoremstyle{plain}
\newtheorem{thm}{\protect\theoremname}
  \theoremstyle{remark}
  \newtheorem{rem}[thm]{\protect\remarkname}
  \theoremstyle{plain}
  \newtheorem*{thm*}{\protect\theoremname}
  \theoremstyle{plain}
  \newtheorem{prop}[thm]{\protect\propositionname}
  \theoremstyle{plain}
  \newtheorem{lem}[thm]{\protect\lemmaname}

\usepackage{babel}

\usepackage{babel}
\providecommand{\lemmaname}{Lemma}
  \providecommand{\propositionname}{Proposition}
  \providecommand{\remarkname}{Remark}
  \providecommand{\theoremname}{Theorem}
\providecommand{\theoremname}{Theorem}

\usepackage{subcaption}

\makeatother

\begin{document}

\title{The scaling limit of critical Ising interfaces is $\mathrm{CLE}_{3}$}

\author{St\'ephane Benoist and Cl\'ement Hongler}
\begin{abstract}
In this paper, we consider the set of interfaces between $+$ and
$-$ spins arising for the critical planar Ising model on a domain
with $+$ boundary conditions, and show that it converges towards nested
CLE$_{3}$.

Our proof relies on the study of the coupling between the Ising model
and its random cluster (FK) representation, and of the interactions
between FK and Ising interfaces. The main idea is to construct an
exploration process starting from the boundary of the domain, to discover
the Ising loops and to establish its convergence to a conformally
invariant limit. The challenge is that Ising loops do not touch the
boundary; we use the fact that FK loops touch the boundary (and hence
can be explored from the boundary) and that Ising loops in turn touch
FK loops, to construct a recursive exploration process that visits
all the macroscopic loops.

A key ingredient in the proof is the convergence of Ising free arcs
to the Free Arc Ensemble (FAE), established in \cite{benoist-duminil-hongler}.
Qualitative estimates about the Ising interfaces then allow one to
identify the scaling limit of Ising loops as a conformally invariant
collection of simple, disjoint $\mathrm{SLE}_{3}$-like loops and
thus by the Markovian characterization of \cite{sheffield-werner}
as a $\mathrm{CLE}_{3}$.

A technical point of independent interest contained in this paper
is an investigation of double points of interfaces in the scaling
limit of critical FK-Ising. It relies on the technology of \cite{kemppainen-smirnov}. 
\end{abstract}

\address{Department of Mathematics, Columbia University. 2990 Broadway, New
York, NY 10027, USA.}

\email{sbenoist@math.columbia.edu}

\address{Chair of Statistical Field Theory, MATHAA Institute, Ecole Polytechnique
F\'ed\'erale de Lausanne, Station 8, 1015 Lausanne, Switzerland. }

\email{clement.hongler@epfl.ch}

\maketitle

\section{Introduction}

\subsection{Schramm-Loewner Evolution}

The introduction of Schramm's SLE curves \cite{schramm-i} opened
the road to decisive progress towards the understanding of 2D statistical
mechanics. The $\left(\mathrm{SLE}_{\kappa}\right)_{\kappa>0}$ form
a one-parameter family of conformally invariant random curves that
are the natural candidates for the scaling limits of interfaces found
in critical lattice models, as shown by Schramm's principle \cite{schramm-i}:
if a random curve is conformally invariant and satisfies the domain
Markov property, then it must be an $\mathrm{SLE}_{\kappa}$ for some
$\kappa>0$. The convergence of lattice model curves to SLE has been
established in a number of cases, in particular for the loop-erased
random walk ($\kappa=2$) and the uniform spanning tree ($\kappa=8$)
\cite{lawler-schramm-werner-ii}, percolation ($\kappa=6$) \cite{smirnov-iv},
the Ising model ($\kappa=3$) and FK-Ising ($\kappa=16/3$) \cite{chelkak-duminil-hongler-kemppainen-smirnov},
and the discrete Gaussian free field $\left(\kappa=4\right)$ \cite{schramm-sheffield}.

The development of SLE has had rich ramifications, in particular the
introduction of the Conformal Loop Ensembles (CLE) \cite{sheffield}.
The $\left(\mathrm{CLE}_{\kappa}\right)_{\kappa\in(\frac{8}{3},8]}$
are conformally invariant collections of $\mathrm{SLE}_{\kappa}$-like
random loops; they conjecturally describe the full set (rather than
a fixed marked set) of macroscopic interfaces appearing in discrete
models. For percolation, the convergence of the full set of interfaces
to $\mathrm{CLE}_{6}$ has been established \cite{camia-newman-ii}.
For the Gaussian free field, the connection with $\mathrm{CLE}_{4}$
is established in \cite{miller-sheffield,aru-sepulveda-werner}. For
the random-cluster (FK) representation of the Ising model, the convergence
of boundary-touching interfaces to a subset of CLE$_{16/3}$ has been
established in \cite{kemppainen-smirnov-ii}. This paper shows convergence
of Ising interfaces to CLE$_{3}$, and this is the first convergence
result in the non boundary touching regime $\kappa\leq4$.

\subsection{Ising Interfaces and SLE}

The Ising model is the most classical model of equilibrium statistical
mechanics. It consists of random configurations of $\pm1$ spins on
the vertices of a finite graph $\mathcal{G}$, which interact with their
neighbors: the probability of a spin configuration $\left(\sigma_{x}\right)_{x\in V}$
is proportional to $\exp\left(-\beta H\left(\sigma\right)\right)$,
where the energy $H\left(\sigma\right)$ is given by $-\sum_{x\sim y}\sigma_{x}\sigma_{y}$
and $\beta$ is a positive parameter called the inverse temperature.

The two-dimensional Ising model (i.e. when $\mathcal{G}\subset\mathbb{Z}^{2}$)
has been the subject of intense mathematical and physical investigations.
A phase transition occurs at the critical value $\beta_{c}=\frac{1}{2}\ln\left(\sqrt{2}+1\right)$:
for $\beta<\beta_{c}$ the spins are disordered at large distances,
while for $\beta>\beta_{c}$ a long range order is present. Thanks
to the exact solvability of the model, much is known about the phase
transition of the model; the recent years have in particular seen
important progress towards understanding rigorously the scaling limit
of the fields \cite{hongler, hongler-smirnov, chelkak-hongler-izyurov, hongler-kytola-viklund} and the interfaces \cite{smirnov-i, chelkak-duminil-hongler-kemppainen-smirnov, benoist-duminil-hongler} of the model at 
the critical temperature $\beta_{c}$.

For the two-dimensional Ising model, we call spin interfaces the curves
that separate the $+$ and $-$ spins of the model (as a technical aside, one needs to make choices when trying to follow an Ising interface on the square lattice, however these discrete choices are irrelevant in the scaling limit). In the case of
Dobrushin boundary conditions (i.e. $+$ spins on a boundary arc and
$-$ spins on the boundary complement) the resulting distinguished
spin interface linking boundary points can be shown to converge to
$\mathrm{SLE}_{3}$ \cite{chelkak-duminil-hongler-kemppainen-smirnov},
using the discrete complex analysis of lattice fermions \cite{smirnov-ii,chelkak-smirnov-ii}.
In the case of more general boundary conditions (in particular free
ones), one obtains convergence to variants of $\mathrm{SLE}_{3}$,
as was established in \cite{hongler-kytola,izyurov}.

Another natural class of random curves are the interfaces of random-cluster
representation of the model (which separate `wired' from `free' regions
in the domain). Following the introduction of the fermionic observables
\cite{smirnov-ii}, it was shown that the Dobrushin random-cluster
interfaces converge to chordal $\mathrm{SLE}_{16/3}$.

The study of more general collections of interfaces for the Ising
model and its FK representation has seen recent progress. With free
boundary conditions, the scaling limit of the interface arcs was obtained
\cite{benoist-duminil-hongler}: by taking advantage of the fact that
such arcs touch the boundary, an exploration tree is constructed,
made of a bouncing and branching version of the dipolar $\mathrm{SLE}_{3}$
process. 

For the FK representation of the Ising model, an exploration tree is constructed in \cite{kemppainen-smirnov-ii},
and this tree allows one to represent the random-cluster loops that
touch the boundary in terms of a branching $\mathrm{SLE}_{16/3}$. 
More recently, the convergence of the full set of these random-cluster 
interfaces to $\mathrm{CLE}_{16/3}$ has been shown in \cite{kemppainen-smirnov-iii}. 

\subsection{Ising Model and CLE}

The purpose of this paper is to rigorously describe the full scaling
limit of the Ising loops that arise in a domain with $+$ boundary
conditions: 
\begin{thm}
\label{thm:main}Consider the critical Ising model on a discretization
$\left(\Omega_{\delta}\right)_{\delta>0}$ of a (simply-connected)
Jordan domain $\Omega$, with $+$ boundary conditions. Then the set
of the interfaces between $+$ and $-$ spins converges in law to a
nested $\mathrm{CLE}_{3}$ as $\delta\to0$, with respect to the metric
$d_{\mathcal{X}}$ on the space of loop collections. 
\end{thm}
The statement we prove is actually slightly stronger (Theorem \ref{thm:main-thm}
in Section \ref{sub:statement-and-strategy}). The precise definition
of the Ising interfaces is given in Section \ref{sub:ising-loops},
the CLE processes are introduced in Section \ref{sub:cle} and the
metric $d_{\mathcal{X}}$ is defined in Section \ref{sub:def-space-loop-collections}.

Our strategy is to identify the scaling limit of these curves by using
the coupling between the Ising model and its random-cluster (FK) representation,
often called Edwards-Sokal coupling \cite{grimmett}. This allows
one to construct an exploration tree describing the Ising loops, by
relying on the recursive application of a two-staged exploration: 
\begin{itemize}
\item We first study the random-cluster interfaces by relying on the fact
that they touch the boundary and hence can be described in terms of
an exploration tree (as in \cite{kemppainen-smirnov-ii}, and similarly
to \cite{benoist-duminil-hongler}). 
\item We then explore the Ising loops, which are contained inside the random-cluster
loops: conditionally on the random-cluster loops, the Ising model
inside has free boundary conditions, allowing one to use the result
of \cite{benoist-duminil-hongler} to identify a subset of the Ising
interfaces. 
\item At the end of the second stage, we obtain a number of a loops. Conditionally
on these loops, the boundary conditions for Ising on the complement
are monochromatic (either completely $+$ or completely $-$), allowing
one to re-iterate the exploration inside of those. 
\end{itemize}
We then show that the conformally invariant interfaces that we have
explored are simple and $\mathrm{SLE}_{3}$-like. Together with the
Markov property inherited from the lattice level, this allows one
to use the uniqueness result of \cite{sheffield-werner} to identify
this limit as $\mathrm{CLE}_{3}$.

For the FK model naturally associated with Ising, the result equivalent to Theorem \ref{thm:main}, namely that the FK interfaces converge to CLE$_{16/3}$, was proved by \cite{kemppainen-smirnov-ii,kemppainen-smirnov-iii}, at least when the domain boundary is analytic. Even though our proof uses a coupling with the FK model to show convergence of the Ising loops, we do not get the joint convergence of FK and Ising interfaces. Another question of interest would be to get a direct proof of the convergence of Ising interfaces to CLE$_3$, i.e., a proof that would do away with using the auxiliary FK model but would instead only rely on the strong properties of CLE to conclude. Such a proof could give a template to use for models beyond Ising.

\subsection{Outline of the Paper}
\begin{itemize}
\item In Section \ref{sec:setup-and-definitions}, we give the definitions
of the graphs, the models, the metrics and the loop ensembles. 
\item In Section \ref{sec:main-theorem}, we give the precise statement
of our main theorem, together with the main steps of the proof. 
\item In Section \ref{sec:scaling-limits-ising-and-fk-ising-interfaces},
we state two results about the scaling limit of Ising and FK interfaces
that are instrumental in our proof, one borrowed from \cite{benoist-duminil-hongler},
and the other one from the Appendix (related to \cite{kemppainen-smirnov-ii}). 
\item In Section \ref{sec:scaling-limit-outermost-loops}, we prove that
the outermost Ising loops have a conformally invariant scaling limit. 
\item In Section \ref{sec:identification-scaling-limit}, we identify the
scaling limit of outermost Ising loops, and then construct the scaling
limit of all Ising loops, thus concluding the proof of the main theorem. 
\item In the Appendix, we study the scaling limit of the FK loops, in particular
proving its existence and conformal invariance, as well as showing
that double points of discrete and continuous FK loops correspond
to each other. 
\end{itemize}

\subsection{Acknowledgements}

The authors would like to thank D. Chelkak, J. Dub\'edat, H. Duminil-Copin,
K. Kyt\"ol\"a, P. Nolin, S. Sheffield, S. Smirnov and W. Werner for interesting
and useful discussions. C.H. gratefully acknowledges the hospitality
of the Courant Institute at NYU, where part of this work was completed, as well as support from the New York Academy of Sciences the Blavatnik Family Foundation, the Latsis Family Foundation, and the ERC Grant CONSTAMIS. C.H. is a member of the SwissMap Swiss NSF NCCR program. 

We thank the anonymous referees for their very helpful comments. 

\section{\label{sec:setup-and-definitions}Setup and Definitions}

\subsection{Graphs}

We consider the usual square grid $\mathbb{Z}^{2}$, with the usual
adjacency relation (denoted $\sim$). We denote by $\left(\mathbb{Z}^{2}\right)^{*}$
the dual graph, by $(\mathbb{Z}^{2})^{m}$ the medial graph (whose
vertices are the centers of edges of $\mathbb{Z}^{2}$; two vertices
of $(\mathbb{Z}^{2})^{m}$ are adjacent if the corresponding edges
of $\mathbb{Z}^{2}$ share a vertex) and by $(\mathbb{Z}^{2})^{b}$
the bi-medial graph (i.e. the medial graph of $(\mathbb{Z}^{2})^{m}$).
Note that $(\mathbb{Z}^{2})^{b}$ has a natural embedding in the plane
as $\frac{1}{2}\mathbb{Z}^{2}+(\frac{1}{4},\frac{1}{4})$. In the following, we will be interested in particular finite subgraphs of $\mathbb{Z}^{2}$, namely those subgraphs that can be constructed as the collection of vertices and edges contained in a given simply-connected finite union of faces of $\mathbb{Z}^{2}$. We will refer to such finite subgraphs as \emph{discrete domains}. For a discrete domain
$\mathcal{G}$, we denote by $\mathcal{G}^{*}$ the dual of $\mathcal{G}$,
by $\mathcal{G}^{m}$ the medial of $\mathcal{G}$ and by $\mathcal{G}^{b}$
the bi-medial of $\mathcal{G}$ (see Figure \ref{fig:ising-graph}, in particular for how these are defined at the boundary). We denote by $\partial\mathcal{G}\subset\mathcal{G}$
the (inner) boundary of $\mathcal{G}$, which we either see as the
set of vertices of $\mathcal{G}$ adjacent to $\mathbb{Z}^{2}\setminus\mathcal{G}$,
or as the circuit of edges separating the faces of $\mathcal{G}$
from those of $\mathbb{Z}^{2}\setminus\mathcal{G}$.

Consider a Jordan domain $\Omega\subset\mathbb{C}$,
i.e. such that its boundary $\partial\Omega$ is a simple closed curve. We call \emph{discretization}
of $\Omega$ a family $\left(\Omega_{\delta}\right)_{\delta}$ of
discrete domains of $\delta\mathbb{Z}^{2}$ (the square grid of mesh
size $\delta>0$) such that $\partial\Omega_{\delta}\to\partial\Omega$
(where we identify $\partial\Omega_{\delta}$ with its edge circuit)
as $\delta\to0$ in the topology of uniform convergence up to reparametrization. Note that any Jordan domain admits a discretization.

\subsection{Ising Model}

The Ising model (see e.g. \cite{grimmett, friedli-velenik} for modern introductions) on a discrete domain $\mathcal{G}$ at inverse temperature $\beta>0$
consists of random configurations $\left(\sigma_{x}\right)_{x\in\mathcal{G}}$
of $\pm1$ spins with probability proportional to $\exp\left(-\beta H\left(\sigma\right)\right)$
where the energy $H$ is given by $-H\left(\sigma\right)=\sum_{x\sim y}\sigma_{x}\sigma_{y}$
(the sum is over all pair of adjacent spins of $\mathcal{G}$). We
will focus on the Ising model at the critical temperature, i.e. with
$\beta=\beta_{c}$. If there are no particular conditions on the spins
of $\partial\mathcal{G}$ we speak of \emph{free boundary conditions},
if the spins of $\partial\mathcal{G}$ are conditioned to be $+1$
(resp. $-1$), we speak of \emph{$+$ boundary conditions} (resp.
\emph{$-$ boundary conditions}).

\subsection{FK Model}

The \emph{Fortuin-Kasteleyn (FK) model} (or random-cluster model, see \cite{grimmett} for
a background reference) on a discrete domain $\mathcal{G}$ is a (dependent)
bond percolation model, which assigns a random \emph{open} or \emph{closed}
state to the edges of $\mathcal{G}$. We then call \emph{configuration} the set of the open edges of $\mathcal{G}$. The $\mathrm{FK}\left(p,q\right)$
model (with parameters $0\leq p \leq 1$, and $q\geq 1$ being a real number) assigns to a configuration $\omega$ a probability proportional
to $p^{\mathbf{o}\left(\omega\right)}\left(1-p\right)^{\mathbf{c}\left(\omega\right)}q^{\mathbf{k}\left(\omega\right)}$,
where $\mathbf{o}\left(\omega\right)$ is the number of open edges,
$\mathbf{c}\left(\omega\right)$ the number of closed edges and $\mathbf{k}\left(\omega\right)$
the number of clusters of $\omega$, i.e. the number of connected components in the subgraph of $\mathcal{G}$ obtained by deleting all the closed edges. The above description defines what is called the FK model with \emph{free boundary conditions}; the FK model with \emph{wired boundary
conditions} if obtained when conditioning all the boundary edges to be open (i.e.
the edges between vertices of $\partial\mathcal{G}$ are forced open).

An important feature of the two-dimensional FK model is \emph{duality} (see \cite[Section 6.1]{grimmett}).
For an FK configuration $\omega$ on a discrete domain $\mathcal{G}$,
we define the \emph{dual configuration $\omega^{*}$} on $\mathcal{G}^{*}$
whose open edges are the dual to the closed edges of $\omega$ and
vice versa. It can be shown that for $p\left(1-p\right)^{-1}p^{*}\left(1-p^{*}\right)^{-1}=q$,
the dual of an $\mathrm{FK}\left(p,q\right)$ configuration on $\mathcal{G}$
with wired boundary conditions is an $\mathrm{FK}\left(p^{*},q\right)$
configuration on $\mathcal{G}^{*}$ with free boundary conditions.
The \emph{self-dual} (or critical, see \cite{BeDC_FK}) FK model corresponds to $\mathrm{FK}\left(p_{\mathrm{sd}},q\right)$,
where the self-dual value $p_{\mathrm{sd}}=\frac{\sqrt{q}}{1+\sqrt{q}}$
is such that $p_{\mathrm{sd}}^{*}=p_{\mathrm{sd}}$.

\subsection{FK-Ising Model }

When $q=2$, the FK model is called the \emph{FK-Ising model}. The
Ising model at inverse temperature $\beta$ can be sampled from the
FK-Ising model with $p=1-e^{-2\beta}$ by performing percolation on
the FK clusters, see e.g. \cite[Section 2.3]{grimmett}: for each FK cluster we
toss a balanced coin and assign the same $\pm1$ spin value (depending on the tossed coin) to all
the vertices of the cluster, and do this independently for each cluster.
The self-dual FK-Ising model with $p_{\mathrm{sd}}=\frac{\sqrt{2}}{1+\sqrt{2}}$
corresponds to the critical Ising model. In this paper, the only FK
model we will work with is the self-dual FK-Ising model. In order
to clearly distinguish it from the Ising model, we will often refer
to the self-dual FK-Ising model as just the FK model.

\subsection{\label{sub:ising-loops}Ising Loops}

A sequence of vertices $v_{1},\ldots,v_{n}$ is called a \emph{strong
path} if $v_{i}\sim v_{i+1}$ for $1\leq i<n$ (where $\sim$ denotes
the adjacency relation) and a \emph{weak path} if $v_{i}$ is weakly
adjacent to $v_{i+1}$ (i.e. $v_{i}$ and $v_{i+1}$ share a face) for $1\leq i<n$.

Consider the Ising model on a discrete domain $\mathcal{G}\subset\mathbb{Z}^{2}$.
An \emph{Ising loop} is an oriented simple loop on $\mathcal{G}^{*}$
(i.e. a closed strong path of $\mathcal{G}^{*}$ such that no edge in the path
is used twice) and such that any edge of the loop has a $+$ spin
on its left and a $-$ spin on its right. In other words, an Ising
loop separates a weak path of $+$ spins and a weak path of $-$ spins,
and is hence clockwise-oriented if it has $+$ spins outside and $-$
spins inside (and counter-clockwise oriented otherwise). An Ising
loop is called \emph{leftmost }if it follows a strong path of $+$
on its left side, and \emph{rightmost} if it follows a strong path
of $-$ spins on its right.

In a domain carrying $+$ boundary conditions, an Ising loop is called
\emph{outermost} if it is not strictly contained inside another Ising
loop, i.e. if it is not separated from the boundary by a a closed
weak path of $-$ spins. Let us now define the \emph{level} of an
Ising loop (in a domain with $+$ boundary conditions). An Ising loop
is said to be of level 1, if it is outermost, of level $2k$ for $k\geq1$
if it is contained inside an Ising loop of level $2k-1$ and if it
is not separated by a weak path of $+$ spins from that loop, and
of level $2k+1$ for $k\geq1$ if it is contained inside of an Ising
loop of level $2k$ and if it is not separated by a weak path of $-$
spins from that loop. Note that two distinct outermost loops can intersect. However, the level of a loop is a well-defined integer: for example, there are not outermost loops of level $2$, as an outermost loop cannot be strictly contained in the interior of another Ising loop.

\subsection{FK Loops and Cut-Out Domains}\label{sec:fkcod}

Given an FK configuration, the set of FK interfaces forms a set of
loops on $\mathcal{G}^{b}$. A bi-medial edge is part of an FK interface
if it lays between a primal FK cluster and a dual FK cluster, i.e.
if it does not cross a primal open edge or a dual open edge. With
wired or free boundary conditions, it is easy to see that the set
of bi-medial edges that are part of an FK interface forms a collection
of disjoint loops (in contrast to Ising loops, which may intersect).

The level of an FK loop is defined by declaring a loop of level 1
or outermost if it is not contained inside another FK loop, and of
level $k>1$ if it is contained in the interior of exactly $k-1$ distinct FK loops. An FK loop of level $k>1$ is hence an outermost FK loop in the interior of an FK loop of level $k-1$.
Note that, as FK loops are disjoint, the interior of two FK loops of same level $k$ are disjoint.

We call the interior of an outermost FK loop a \emph{cut-out
domain}. This notion will be crucial for us in the scaling limit.
The set of FK loops satisfies the following spatial Markov property
(see \cite[Theorem 3.4]{grimmett}): consider the FK model with wired boundary conditions,
conditionally on the outermost FK loops, the model inside the cut-out
domains consists of independent FK models with free boundary conditions. 
\begin{rem}
\label{rem:ising-loops-subset-dual-fk}In the coupling with FK, the
Ising loops are always a subset of the dual FK configuration. In particular
Ising loops and FK loops never cross. As a consequence, all Ising loops
are contained in the cut-out domains of the corresponding FK configuration. 
\end{rem}
\begin{figure}[!ht]
\centering \includegraphics[width=12cm]{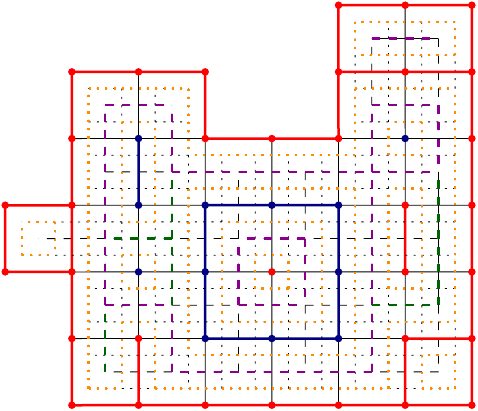} \caption{Ising and FK loops. Plain lines represent the domain $\mathcal{G}$,
dashed lines represent $\mathcal{G}^{*}$ and dotted lines represent
$\mathcal{G}^{b}$. In red and blue, the primal FK configuration $\omega$
($k(\omega)=7$), red corresponding to FK connected components carrying
$+$ Ising spins, and blue to FK components carrying $-$ spins. The
green and purple represent the dual FK configuration, with purple being the subset traced by Ising loops. In orange, the set of
FK loops.}
\label{fig:ising-graph}
\end{figure}

\subsection{\label{sub:def-space-loop-collections}The Space
of Loop Collections}
An oriented loop $\gamma$ is an equivalence class of continuous maps from the unit circle $\mathbb{S}^1$ to the plane $\mathbb{R}^2$, where the equivalence is given by orientation-preserving reparametrizations.

Consider the metric $d_{\Gamma}$ on the space of oriented loops defined
as the supremum norm up to reparametrization: $d_{\Gamma}\left(\gamma,\tilde{\gamma}\right)=\inf\|\gamma-\tilde{\gamma}\|_{\infty}$,
where the infimum is taken over all orientation-respecting parametrizations
of the loops $\gamma$ and $\tilde{\gamma}$. We define the space
$\Gamma$ to be the completion of the set of simple oriented loops
for the metric $d_{\Gamma}$: $\Gamma$ is the set of oriented \emph{non-self-crossing}
loops (including trivial loops reduced to points).

The space $\mathcal{X}$ of loop collections is the space of at most
countable collections $\left\{ \gamma_{i}\right\} _{i\in I}$ of elements of $\Gamma$
(loops can appear with multiplicity, and we include the empty collection), such that 
\begin{itemize}
\item For each $i\in I$, the loop $\gamma_{i}$ is not reduced to a point. 
\item For each scale $\varepsilon >0$,
the set of indices $\left\{ i\in I: \mathrm{diam}\left(\gamma_{i}\right)\geq\varepsilon \right\} $
is finite, where $\mathrm{diam}$ denotes the Euclidean diameter.
\end{itemize}

We now define a $\sigma$-algebra on $\mathcal{X}$. A matching of two sets $I$ and $J$ will denote a subset $\pi \subseteq I \times J$ such that for each $i\in I$ there is at most one $j \in J$ such that $(i,j)\in\pi$, and reciprocally, for each $j\in J$ there is at most one $i \in I$ such that $(i,j)\in\pi$. Given a matching $\pi$, we denote by $I^\pi$ (resp. $J^\pi$) the set of unmatched indices of $I$ (resp. $J$), i.e. the subset of indices $i \in I$ (resp. $j\in J$) such that for all $j\in J$ (resp. $i\in I$), $(i,j)\notin \pi$.

We work with the Borel $\sigma$-algebra on $\mathcal{X}$ associated to the following metric: 
\[
d_{\mathcal{X}}\left(\left\{ \gamma_{i}\right\} _{i\in I},\left\{ \gamma'_{j}\right\} _{j\in J}\right)=\inf_{\pi}\max\left(\sup_{\left(i,j\right)\in\pi}d_{\Gamma}\left(\gamma_{i},\gamma'_{j}\right),\sup_{i\in I^\pi}\mathrm{diam}\left(\gamma_{i}\right),\sup_{j\in J^\pi}\mathrm{diam}\left(\gamma'_{j}\right)\right),
\]
where the infimum is taken over all matchings $\pi$ of the index sets $I$ and $J$.

\subsection{The Interior of a Non-Self-Crossing Loop}

The following Lemma gives a number of useful facts about non-self-crossing loops

\begin{lem}\label{lem:cut-out}
The connected components of the interior of a non-self-crossing loop are open sets homeomorphic to discs, whose boundaries can be traced by continuous curves.
\end{lem}
We call these connected components \emph{cut-out domains}. The boundary of a cut-out domain is a curve, which is not necessarily simple. We will make convergence statement about cut-out domains, these will always be for the topology of uniform convergence up to reparametrization for the boundary curves.

\begin{proof}
Let us first give a formal definition of what it means for a connected component of the complement of a loop $\gamma$ to be in the interior of $\gamma$. Consider a continuous family of simple smooth curves $\gamma_t$ that converge to $\gamma$ when $t\rightarrow 0$. Without loss of generality, we can assume that all the curves $\gamma_t$ and $\gamma$ are positively oriented. Given a point $z$ in the complement of $\gamma_t$, the line integral
\begin{equation}\label{eq:int}
I_t(z)=\frac{1}{2\pi i}\oint_{\gamma_t} \frac{1}{w-z} dw
\end{equation}
takes the value $1$ or $0$ depending on whether the point $z$ is inside or outside $\gamma_t$. For any point $z\notin \gamma$, the quantity $I_t(z)$ is eventually well-defined (as $t \downarrow 0$), and as it is continuous in $t$, needs to be constant. We obtain a limiting value $I(z)= 1\ {\rm or}\ 0$, defined off $\gamma$, and locally constant where defined, hence constant on the connected components of the complement of $\gamma$. We call such a component interior if $I=1$ on it, and exterior otherwise.

We now prove that connected components of the complement of a non-self-crossing loop $\gamma$ in the compactified plane $\widehat{\mathbb{R}^2}=\mathbb{S}^2$ are homeomorphic to open discs. Indeed, consider a sequence of simple loops $\gamma^n\rightarrow\gamma$, and a point $z\notin\gamma$. Let $D$ be the connected component of $\mathbb{S}^2\setminus\gamma$ containing $z$. Consider the uniformization map $\phi^n$ from the unit disk $\mathbb{D}$ to the connected component of $\mathbb{S}^2\setminus\gamma^n$ containing $z$, normalized so that $0$ is sent to $z$ and so that the derivative there is a positive real. These maps converges uniformly on compact subsets of the disc to a conformal map $\phi$ (by Caratheodory's theorem \cite[Theorem 1.8]{Pom}), and the image $\phi(\mathbb{D})=D$, is hence the image of a disk by a one-to-one bicontinuous map, so is itself an open set homeomorphic to the disk. 

Finally, we show that the boundary of the cut-out domain $D$ containing $z$ can be traced by a continuous curve. We first show that, given another cut-out domain $D_w$ containing the point $w$, we can construct a subloop $\widetilde{\gamma}$ of $\gamma$ whose interior is a reunion of interior connected components of $\gamma$, that include $D$ but not $D_w$.

We consider an approximation by simple curves $\gamma^n\rightarrow\gamma$. For $n$ large enough, the points $z$ and $w$ are interior to $\gamma^n$. We call $\eta^n$ the infimum of the diameters of curves joining two points of $\gamma^n$ and otherwise staying in its interior that separate $z$ from $w$ in the interior of $\gamma^n$. We then pick such a curve, of diameter no more than $2\eta^n$ and call its boundary points $\gamma^n(u^n)$ and $\gamma^n(v^n)$.  Up to extracting a subsequence, we can assume by compactness that $u^n\rightarrow u$ and $v^n\rightarrow v\neq u$. Note that $\eta^n\rightarrow 0$ as $z$ and $w$ belong to different connected components in the limit. Hence $\gamma(u)=\gamma(v)$ is a double point of the curve, and moreover, by construction, the two corresponding subloops of $\gamma$, say $\gamma_1$ and $\gamma_2$, are positively oriented, non-self-crossing, and do not cross each other. By looking at the integral formula (\ref{eq:int}) to determine whether a point is surrounded by a loop, and using that $I^\gamma=I^{\gamma_1}+I^{\gamma_2}$, we see that each interior cut-out domain of $\gamma$ is either in the interior of $\gamma_1$ or in the interior of $\gamma_2$. Moreover it is clear that one of these loops, say $\gamma_1$, surrounds $D_w$, and the other, $D$. This provides the subloop $\widetilde{\gamma}=\gamma_2$ as claimed.

Enumerating the interior cut-out domains of $\gamma$ using points they contain $z,w_1,w_2,\cdots, w_l,\cdots$, we can iteratively extract subloops $\gamma^l$ of $\gamma$ that contain $D$ in their interior but not $D_{w_1}, \cdots , D_{w_l}$. The curves $\gamma^l$ converge up to reparametrization (at the very least it is easy to see that we can assume convergence up to extracting a subsequence). The limiting loop $\gamma_D$ is by construction a positively oriented non-self-crossing loop whose interior is $D$.
\end{proof}

\subsection{Nested and Non-Nested Loop Collections}\label{sec:nested}

We say that two loops $\gamma$ and $\gamma'$ of $\Gamma$ \emph{do not cross each other} if we can find two sequences $(\gamma_n)_{n\in\mathbb{N}}$ and $(\gamma'_n)_{n\in\mathbb{N}}$ of elements of $\Gamma$ such that $\gamma_n\rightarrow\gamma$, $\gamma'_n\rightarrow\gamma'$, and for each $n\in \mathbb{N}$, the loops $\gamma_n$ and $\gamma'_n$ are disjoint.

\begin{lem}
Given two distinct non-self-crossing loop that do not cross each other, then either
\begin{itemize}
\item their interiors are disjoint, or
\item the interior of one of the loops, say $\gamma$, is included in the interior of the other loop $\gamma'$.
\end{itemize}
\end{lem}

In the second case, we say that the loop $\gamma$ is \emph{nested} in $\gamma'$.
\begin{proof}
We omit the proof of this simple result, which can formally be proven using (\ref{eq:int}).
\end{proof}
A \emph{non-crossing} collection of loops is a loop collection of $\mathcal{X}$ such that no two loops cross each other.
A \emph{non-nested} collection of loops is a loop collection of $\mathcal{X}$ that is non-crossing and such that no loop is nested in another.

Given a non-crossing loop collection such that no two loops are equal, we define the \emph{level} of a loop as the number of loops containing it, plus $1$. Note that the level of a loop is always finite, by the diameter constraint on loop collections. \emph{Outermost loops} are loops of level 1, or alternatively, loops that are not nested in any other loop. Similarly, \emph{innermost loops} are loops such that no other loop is nested inside of them. Given a non-crossing collection of loops $\mathcal{C}$, we call \emph{cut-out domains of} $\mathcal{C}$ the cut-out domains of the innermost loops of $\mathcal{C}$.

This general definition for the level of a loop does not apply to discrete Ising loops (which is not a non-crossing collection, see Section \ref{sub:ising-loops}), but coincides with the definitions we give for discrete FK loops in Section \ref{sec:fkcod}.

\subsection{Measurability of Some Loop Collections}

\begin{lem}\label{lem:top}
The space $\mathcal{X}$ is complete and separable, hence is a Polish space. Moreover, the following
events on $\mathcal{X}$ are measurable for the Borel $\sigma$-algebra: 
\begin{itemize}
\item The collection $\left(\gamma_{i}\right)_{i\in I}$ is non-crossing.  
\item The collection $\left(\gamma_{i}\right)_{i\in I}$ is non-nested. 
\item The loops of the collection $\left(\gamma_{i}\right)_{i\in I}$ are
disjoint. 
\item All the loops of $\left(\gamma_{i}\right)_{i\in I}$ are simple. 
\end{itemize}
\end{lem}

\begin{proof}
The set of finite collections of simple loops made of edges of one of  the lattices $2^{-n}\mathbb{Z}^2$ forms a countable family which is dense in $\mathcal{X}$. In other words, $\mathcal{X}$ is separable.

We now show that $\mathcal{X}$ is complete. Let $\mathcal{C}_n$ be a Cauchy sequence in $\mathcal {X}$. Let us call $N_n(\varepsilon)$ the number of loops of $\mathcal{C}_n$ that are of diameter larger than or equal to $\varepsilon$. $N_n$ is a non-increasing integer-valued left-continuous function that goes to $0$ as $\varepsilon$ goes to $\infty$. From this, we can see that $N_n(\varepsilon)$ converges pointwise to a function $N$ as $n\to \infty$, except maybe at the jump points (i.e., discontinuity points) of the limiting function $N$. Let us consider a sequence of sizes $\varepsilon_i\to 0$ whose elements are distinct from the jump points of $N$. For a fixed $\varepsilon_i$, for any $n,m$ large enough (say $n,m\geq n_0$), the collections $\mathcal{C}_n$ and $\mathcal{C}_m$ will have the same number of loops of diameter larger than $\varepsilon_i$. Moreover, provided that $n,m$ are large enough, any matching between the loops of $\mathcal{C}_n$ and $\mathcal{C}_m$ that is close to providing the optimal matching distance $d_{\mathcal{X}}(\mathcal{C}_n,\mathcal{C}_m)$ will have to match all of the loops of diameter larger than $\varepsilon_i$ with each other. From this, we see that we can match consistently (for all $n\geq n_0$) the $N(\varepsilon_i)$ large loops in such a way that their $d_\Gamma$ distance goes uniformly to $0$. The fact that the space $\mathcal{X}$ is complete hence follows from the fact that $\Gamma$ is.

A number of sets of loop collections can then be shown to me measurable:
\begin{itemize}
\item The set of collections that consists of non-crossing loops is a closed set of $\mathcal{X}$, hence is measurable. The same holds for non-nested collections.
\item For each $\varepsilon,\delta >0$, let us consider the set of collections $D_{\varepsilon,\delta}$ such that the open $\delta$-neighborhoods of all the loops of diameter strictly larger than $\varepsilon$ are all disjoint. The set $D_{\varepsilon,\delta}$ is closed, and we can write the set $D$ of collections such that all loops are disjoint as $D=\cap_\epsilon \cup_\delta D_{\varepsilon,\delta}$. Hence $D$ is a measurable set.

\item We say that a loop $\gamma$ is in the set $\Sigma_\delta$ of approximately simple loops at scale $\delta$ if any of its double points cuts up the loop $\gamma$ in two pieces that are not both of diameter larger than or equal to $\delta$. Note that $\Sigma_\delta$ is an open set.
For each $\varepsilon,\delta >0$, let us consider the set of collections $S_{\varepsilon,\delta}$ such that all the loops of diameter larger than or equal to $\varepsilon$ belong to $\Sigma_\delta$. The set is $S_{\varepsilon,\delta}$ is open, and we can write the set $S$ of collections such that all loops are simple as $S=\cap_\epsilon \cap_\delta S_{\varepsilon,\delta}$. Hence $S$ is a measurable set.
\end{itemize}

\end{proof}

\subsection{\label{sub:cle}Conformal Loop Ensembles}

The CLE measures have been introduced in \cite{sheffield} as the
natural candidates to describe conformally invariant collections of
loops arising as scaling limits of statistical mechanics interfaces.
They form a family indexed by $\kappa\in(\frac{8}{3},8]$ of random
collections of $\mathrm{SLE}_{\kappa}$-like loops.

The usual $\mathrm{CLE}_{\kappa}$ measure (as opposed to nested $\mathrm{CLE}_{\kappa}$, defined below) is defined on a planar
simply-connected domain and consists of a collection of non-nested
loops. For $\kappa\in(\frac{8}{3},4]$, the usual $\mathrm{CLE}_{\kappa}$
have an elegant loop-soup construction \cite[Theorem 1.6]{sheffield-werner}: the
$\mathrm{CLE}_{\kappa}$ loops can be constructed by taking the outer boundaries
of clusters of loops in a Brownian loop soup of intensity $c=\left(3\kappa-8\right)\left(6-\kappa\right)/2\kappa$.

We now iteratively define a random loop collection called \emph{nested} $\mathrm{CLE}_{\kappa}$. Its outermost loop (or level $1$) have the law of a usual $\mathrm{CLE}_{\kappa}$. Given its loops of level $k\geq 1$, we define its loops of level $k+1$ as independent samples of usual $\mathrm{CLE}_{\kappa}$ in each of the cut-out domains associated to the loops of the level $k$. We choose to orient CLE loops according to their level: clockwise
for odd level loops, and counterclockwise for even level loops.

\subsection{\label{sub:markovian-characterization}CLE Markovian Characterization}

An important result on CLEs is their Markovian characterization. 
\begin{thm*}[{\cite[Theorem 1.4 and Section 2.1]{sheffield-werner}}]
A family of measures $\mu_{\Omega}$ on non-nested collections of
simple loops defined on simply-connected domains $\Omega$ is the
usual $\mathrm{CLE}_{\kappa}$ for a certain $\kappa\in(\frac{8}{3},4]$
if and only if the following holds with probability $1$: 
\begin{itemize}
\item The collection is locally finite: for any $\varepsilon>0$ and any bounded
region $K\subseteq\Omega$, there are only finitely-many loops of
diameter greater than $\varepsilon$ in $K$. 
\item Distinct loops of the collection are disjoint. 
\item The family is conformally invariant: for any conformal mapping $\varphi:\Omega\to\varphi\left(\Omega\right)$,
we have $\varphi_{\ast}\mu_{\Omega}=\mu_{\varphi\left(\Omega\right)}$. 
\item The family satisfies the \emph{Markovian restriction property} on
any simply-connected domain $\Omega$: for any compact set $K\subset\Omega$
such that $\Omega\setminus K$ is simply-connected, if $\left\{ \gamma_{i}\right\} _{i\in I}$
is sampled from $\mu_{\Omega}$, setting $I_{K}:=\left\{ i:\gamma_{i}\cap K\neq\emptyset\right\} $
we have that $\left\{ \gamma_{i}\right\} _{i\in I}$ conditioned on
$\left\{ \gamma_{i}\right\} _{i\in I_{K}}$ has the law of $\mu_{\Omega\setminus\mathcal{L}_{K}}$,
where $\mu_{\Omega\setminus\mathcal{L}_{K}}$ is defined as the independent
product of $\mu_{\Omega'}$ taken over all connected components $\Omega'$
of $\Omega\setminus\mathcal{L}_{K}$ and where $\mathcal{L}_{K}=K\cup\left\{ \mathrm{Inside}\left(\gamma_{i}\right)\cup\gamma_{i}:i\in I_{K}\right\} $. 
\end{itemize}
\end{thm*}

\section{\label{sec:main-theorem}Main Theorem}

\subsection{\label{sub:statement-and-strategy}Statement and Strategy}

Let us now give the precise version of our main result: 
\begin{thm}
\label{thm:main-thm}Consider the critical Ising model with $+$ boundary
conditions on a discretization $\left(\Omega_{\delta}\right)_{\delta>0}$
of a Jordan domain $\Omega$,. Then as the mesh size $\delta\to0$,
the set of all leftmost Ising loops converges in law with respect
to $d_{\mathcal{X}}$ to the nested $\mathrm{CLE}_{3}$.

Furthermore, for any $\varepsilon>0$, the following holds with probability tending to $1$ as $\delta\to0$: for any Ising loop $\ell$ of diameter larger than $\varepsilon$, there exists
a leftmost loop $\ell_{L}$ such that $ d_\Gamma ( \ell, \ell^L ) \leq \varepsilon $ and such that the connected components of
$\left(\ell\cup\ell^{L}\right)\setminus\left(\ell\cap\ell^{L}\right)$
have diameter less than $\varepsilon$. \end{thm}

The second part of the statement (that will be proved as Lemma \ref{lem:non-left-most-loop}) tells us that in order to understand all Ising loops, it is enough to understand leftmost Ising loops. Indeed, any macroscopic Ising loop $\ell$ is close to a leftmost loop $\ell_{L}$ in a very strong sense: not only are the loops close in the topology of uniform distance up to reparametrization, but the edges they share form a dense subset of each loop.

\begin{rem}
The same result holds for rightmost loops, as the proof will show. 
\end{rem}
The strategy for the proof, illustrated in Figure \ref{fig:fig}, is the following: 
\begin{itemize}
\item We first prove that the collection of level one (i.e. outermost) Ising
loops has a conformally invariant scaling limit (Section \ref{sec:scaling-limit-outermost-loops}). 
\item We then show that this limit consists of loops that are simple, do
not touch the boundary or each other, and satisfy the Markovian restriction
property (Section \ref{sub:qualitative-scaling-limit}). 
\item We then use the characterization of CLE to identify the scaling limit
of the outermost loops as non-nested $\mathrm{CLE}_{3}$ and finally
obtain the convergence of all Ising loops to nested $\mathrm{CLE}_{3}$
(Section \ref{sub:proof-main-thm}).\end{itemize}
\begin{rem}
In \cite{miller-sheffield-werner}, the authors explain how $\mathrm{CLE}_{3}$
can be obtained by performing a percolation
on the $\mathrm{CLE}_{16/3}$ clusters (a procedure analogous to the one used in the discrete to construct the Ising model from the FK model). This approach explains how the joint coupling works
in the continuum and provides a proof scheme for the joint convergence
of Ising and FK loops towards a coupling of CLE$_{16/3}$ and CLE$_{3}$
(our approach does not, as we keep resampling the coupled FK model
to further explore Ising loops). Remark \ref{rem:fktouches} provides
some of the technical tools needed for this joint convergence, but
some further study of the set of discrete FK loops seems needed in
order to get a complete argument. 
\end{rem}
\begin{figure}
\begin{subfigure}{.49\textwidth} \centering \includegraphics[width=1\linewidth]{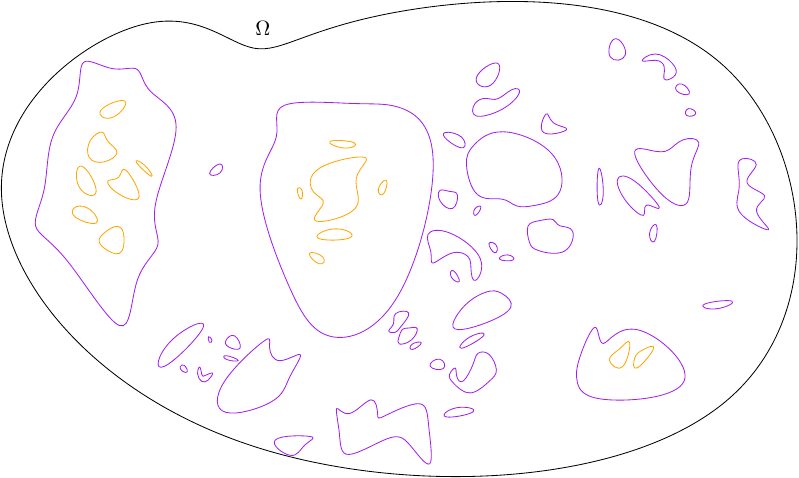}
\caption{The Ising loops to explore (colors for Ising loops: level 1 in purple,
level 2 in orange)}
\label{fig:sfig1} \end{subfigure} \begin{subfigure}{.49\textwidth}
\centering \includegraphics[width=1\linewidth]{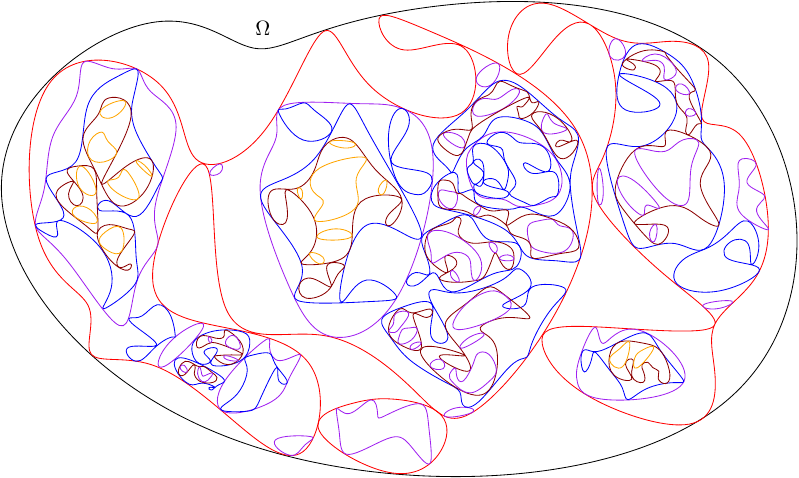}
\caption{Ising loops and the coupled FK (colors for FK loops: level 1 in red,
level 2 in blue, level 3 in brown)}
\label{fig:sfig3} \end{subfigure} \begin{subfigure}{.49\textwidth}
\centering \includegraphics[width=1\linewidth]{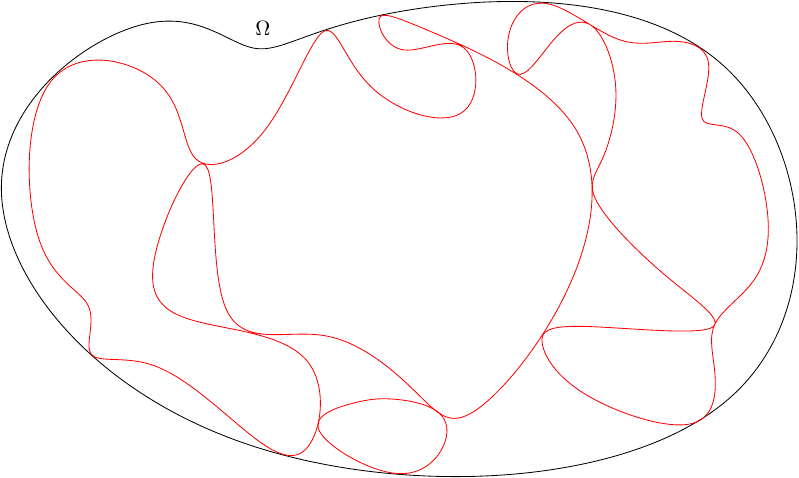}
\caption{We first consider all FK loops of level 1}
\label{fig:sfig4} \end{subfigure} \begin{subfigure}{.49\textwidth}
\centering \includegraphics[width=1\linewidth]{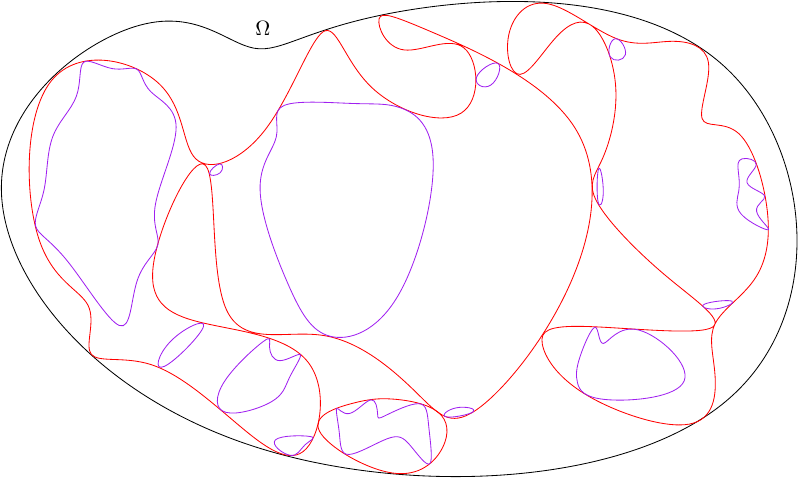}
\caption{Then, we find all Ising loops (of level 1) that touch the explored
FK loops}
\label{fig:sfig5} \end{subfigure} \begin{subfigure}{.49\textwidth}
\centering \includegraphics[width=1\linewidth]{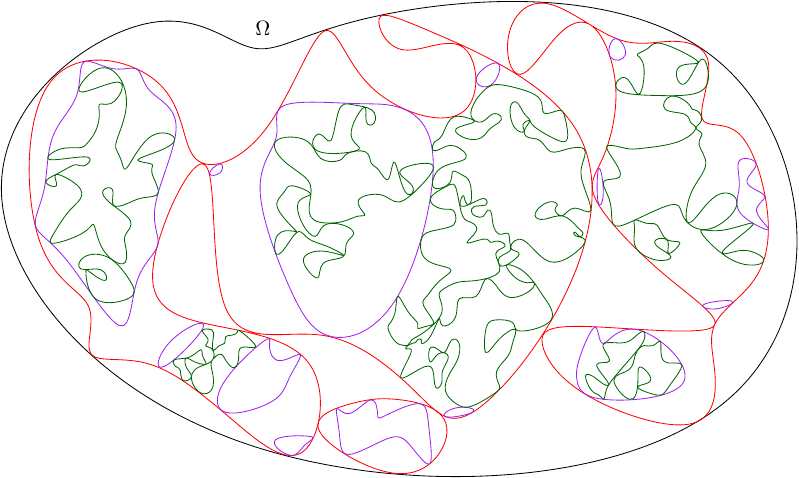}
\caption{We resample the FK loops (in green) in the domains bounded by the
Ising loops we have found}
\label{fig:sfig6} \end{subfigure} \begin{subfigure}{.49\textwidth}
\centering \includegraphics[width=1\linewidth]{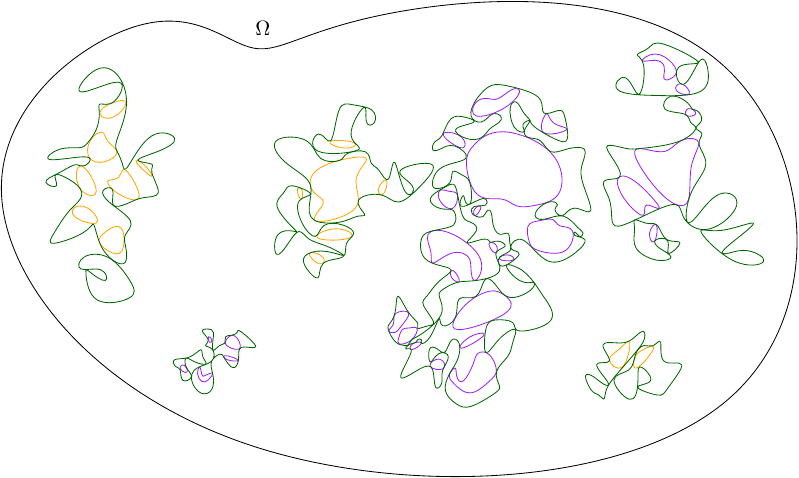}
\caption{We find all Ising loops (of level 1 and 2) that touch the resampled
outermost FK loops}
\label{fig:sfig7} \end{subfigure} \begin{subfigure}{.49\textwidth}
\centering \includegraphics[width=1\linewidth]{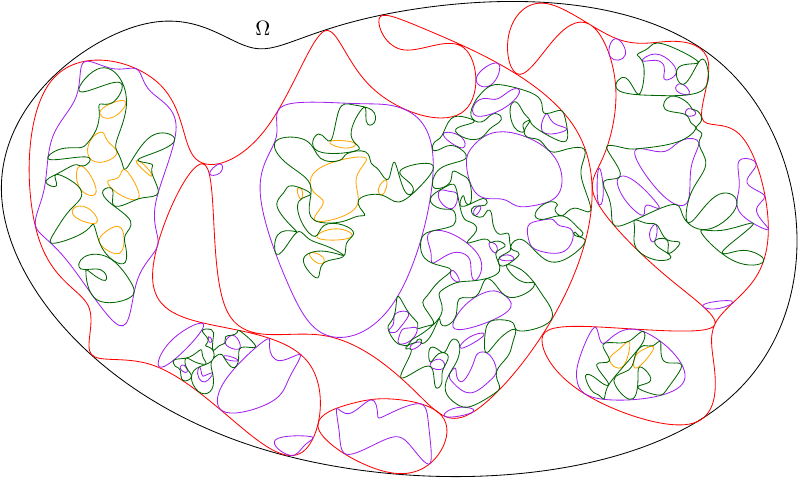}
\caption{All the loops used in the exploration process}
\label{fig:sfig8} \end{subfigure} \caption{The exploration scheme}
\label{fig:fig} 
\end{figure}

\section{\label{sec:scaling-limits-ising-and-fk-ising-interfaces}Scaling
Limits of Ising and FK Interfaces}

In this section, we state two results on which our proof relies: first,
the identification of the scaling limit of the free boundary conditions
arc for the Ising model and second, the conformal invariance of the
scaling limit of the FK interface loops.

\subsection{Ising Free Arc Ensembles}

The first result that we need is the identification of the scaling
limit of the Ising arcs that arise with free boundary conditions.
For the Ising model on a discrete domain $\mathcal{G}$ with free
boundary conditions, we call an \emph{Ising arc} a spin interface
that links two boundary points. In the continuous, we refer to the
set of arcs produced by a branching SLE$_{3}(-\frac{3}{2},-\frac{3}{2})$
as the Free Arc Ensemble (FAE) \cite{benoist-duminil-hongler}. 
\begin{thm}[{{{\cite[Theorem 6]{benoist-duminil-hongler}}}}]
\label{thm:ising-main-theorem} Consider the critical Ising model
on a discretization $\left(\Omega_{\delta}\right)_{\delta>0}$ of
a Jordan domain $\Omega$, with free boundary conditions. Then as
$\delta\to0$, the set of all Ising arcs converges in law to the Free
Arc Ensemble (for the Hausdorff metric on sets of curves, where curves
are equipped with the topology of uniform convergence up to reparametrization). 
\end{thm}
As explained in \cite{benoist-duminil-hongler}, the scaling limits
of the interfaces linking pairs of boundary points, and hence boundary touching loops, can be deterministically recovered by gluing
the FAE arcs. These two sets of curves (the arcs on the one hand and
the scaling limit of the boundary touching loops) contain the same data in the continuous.

Let us describe how boundary touching loops can be recovered from Ising arcs. We formally put $+$ spins on the boundary of a discrete domain $\mathcal{G}$ and consider the free Ising model inside of $\mathcal{G}$: the formal boundary spins do not interact, but they play a role in determining what we call an Ising interface. Given a spin configuration $\sigma$, we will consider the \emph{spin-flip} of $\sigma$ which is the spin configuration obtained by switching the value of all the spins inside of $\mathcal{G}$, except for the formal boundary spins that stay at their fixed $+$ value.

Given a face $z\in \mathcal{G}$ and an edge $b$ on the boundary of $\mathcal{G}$ (i.e., an edge that lies between a fixed $+$ spin and a free spin), let  $\mathcal{A}(z,b)$ be the set of all the leftmost
Ising arcs (i.e., leftmost Ising interfaces joining two points of the boundary of $\mathcal{G}$) that separate $z$ from $b$ in $\mathcal{G}$. The proximity to $z$ gives a natural ordering of the set $\mathcal{A}(z,b)$: given two distinct arcs of $\mathcal{A}(z,b)$, one of them is always closer to $z$, in the sense that it separates the other arc from $z$. When the set $\mathcal{A}(z,b)$ is non-empty, we define $a(z,b)$ to be the arc closest to $z$ among all elements of this set, otherwise, we let $a(z,b)=b$. We call the \emph{inside} (resp. the \emph{outside}) of $a(z,b)$ the set of spins neighboring the arc $a(z,b)$ on the side of $z$ (resp. on the side of $b$). 

\begin{figure}
\centering \includegraphics[width=8cm]{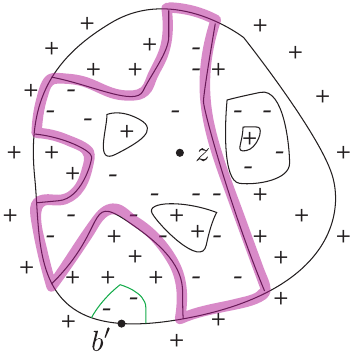}
\caption{
The arcs and loops of an Ising configuration in $ \mathcal{G} $. The path outlined  in purple is the concatenation of the arcs $ a (z,b) $ over all the boundary edges $ b $, as in Lemma \ref{lem:arcsandloops}. The green arc belongs to  $\mathcal{A}(z,b')$ but is not the arc $a (z,b') $ .}
\label{fig:cle-loops-from-arcs}
\end{figure}

We now explain how to recover from the Ising arcs the set $\mathcal{L}\left(\mathcal{G}\right)$ of all the leftmost Ising loops that touch the boundary of $\mathcal{G}$ (see also Figure \ref{fig:FAE}). Note that this construction relies on the formal $ + $ spins on the boundary of $ \mathcal{G} $.
\begin{lem}\label{lem:arcsandloops}
For any spin configuration $ \sigma $ and any face $z\in\mathcal{G}$, denoting by $ \ell $ the concatenation of the arcs $ a (z, b) $ over all the boundary edges $ b $ (see Figure \ref{fig:cle-loops-from-arcs}), we have
\begin{itemize}
\item Either, for any boundary edge $b$, the inside of $a(z,b)$ consists of $-$ spins only. In that case $ \ell $ is a loop in $\mathcal{L}\left(\mathcal{G}\right)$.
\item Or, for any boundary edge $b$, the inside of $a(z,b)$ consists of $+$ spins only. In that case, $\ell$ bounds a connected component $R$ of
\[
\widetilde{\mathcal{G}}= \mathcal{G} \setminus \bigcup_{\gamma\in\mathcal{L}(\mathcal{G})} \gamma
\]
that is not surrounded by a loop in $\mathcal{L}(\mathcal{G})$. Moreover, $\ell$ is an Ising loop for the spin-flipped configuration.
\end{itemize}
\end{lem}

\begin{proof}
Let us fix $z\in \mathcal{G}$, and assume that there exists an edge $b'$ such that $a(z,b')$ carries $-$ spins on its inside. Then the (weak) connected component $\mathcal{G}'$ of $-$ spins attached to the arc $a(z,b')$ disconnects $z$ from the boundary (as the arc $a(z,b')$ is extremal). Hence the outer boundary of $\mathcal{G}'$ consists of the arcs $a(z,b)$ where $b$ ranges over all boundary edges, which forces these arcs have to carry $-$ spins on their inside. 

If the face $z$ is such that the above assumption does not hold, then all the edges of $a(z,b)$ carry $+$ spins on their inside. It is straightforward that their concatenation bounds a connected component $R$ of $\widetilde{\mathcal{G}}$ that is not surrounded by a loop in $\mathcal{L}(\mathcal{G})$. Note that given an edge $b$ such that the set $\mathcal{A}(z,b)$ is non-empty, the arc $a(z,b)$ carry $-$ spins on its outside and $+$ spins on its inside, whereas if $\mathcal{A}(z,b)$ is empty, $a(z,b)=b$ is a boundary edge lying between two $+$ spins. If we flip all the spins in the interior of $\mathcal{G}$, any edge of the boundary of $R$ will carry $-$ spins on its inside and $+$ spins on its outside. In other words the boundary of the region $R$ is an Ising loop for the spin-flipped configuration.
\end{proof}

\subsection{Conformal Invariance of FK-Ising Interfaces and Cut-Out Domains }

For the proof of our main theorem, we need the following result, which
is closely related to (but independent of) the result of Kemppainen
and Smirnov about the scaling limit of the boundary-touching FK loops
in smooth domains \cite{kemppainen-smirnov-ii}. This result includes
in particular the convergence of FK cut-out domains to \emph{continuous
cut-out} \emph{domains}, defined as the maximal domains contained
inside the scaling limit of FK loops. 
\begin{prop}
\label{prop:fk-loops-scaling-limit}Consider the critical FK-Ising
loops on a discretization $\left(\Omega_{\delta}\right)_{\delta}$
of a Jordan domain $\Omega$, with wired boundary conditions. The
FK loops have a conformally invariant scaling limit (in law, with respect to $d_\mathcal{X}$) as the mesh size
$\delta\to0$. This scaling limit is almost surely a non-crossing collection. Furthermore the discrete cut-out domains of level one FK loops converge to
the cut-out domains of the outermost loops of this scaling limit. \end{prop}
\begin{proof}
The first part is proven as Proposition \ref{prop:cvallFK} in the
Appendix and the second part follows from Remark \ref{rem:fktouches}
just after. 
\end{proof}

\section{\label{sec:scaling-limit-outermost-loops}Scaling Limit of Ising
Outermost Loops }

We start with a technical lemma to control the diameters of Ising loops: we show it is impossible to find an arbitrarily large collection of macroscopic Ising loops.

Given a collection $\mathcal{C}$ of loops, let $\underline{d}_{\mathcal{C}}$ (resp. $\overline{d}_{\mathcal{C}}$) denote the infimum (resp. supremum) of the diameters of the loops in $\mathcal{C}$. Consider the critical Ising model with $+$
boundary conditions on a discretization $\left(\Omega_{\delta}\right)_{\delta>0}$ of
a Jordan domain $\Omega$. For any integer $n\geq 1$, we let $\underline{d}_n$ denote the supremum of the quantities $\underline{d}_{\mathcal{C}}$, where $\mathcal{C}$ ranges over all collections of $n$ disjoint Ising loops.

\begin{lem}\label{lem:finitelymanybigloops}
The quantity $\underline{d}_n$ converges in probability to $0$ as $n\to\infty$, uniformly in the mesh size $\delta$:
\[
\forall \varepsilon_1,\varepsilon_2>0,\ \exists n_0\in\mathbb{N},\ \forall \delta>0,\ \forall n\geq n_0,\  \mathbb{P}\left[\underline{d}_n\geq \varepsilon_1 \right] \leq \varepsilon_2.
\]
\end{lem}

\begin{proof}
By contradiction, if this were not to the case, we could find $\varepsilon_1, \varepsilon_2>0$ such that for any integer $n$, we could find a mesh size $\delta_n$ such, with probability at least $\varepsilon_2$, that there would be a collection $\mathcal{C}$ of $n$ disjoint Ising loops such that $\underline{d}_{\mathcal{C}} > \varepsilon_1$. Note that $\delta_n\to 0$ as, for each fixed $\delta$, one can draw only finitely many simple loops on the graph $\Omega_\delta$.

Given any scale $\eta<\varepsilon_1$, we can find a finite collection of annuli of inner radius $\eta$ and outer radius $\varepsilon_1$ such that the domain $\Omega$ is covered by the balls of radius $\eta$ at the center of these annuli. Moreover, we can pick such a covering collection by using a number $C \eta^{-2}$ of annuli, where $C$ is a constant that depends on the domain $\Omega$ but not on $\eta$. Each Ising loop has to intersect the inside of at least one of these annuli (as they form a cover of our domain), and so each Ising loop of diameter larger than $\varepsilon_1$ forces the existence of an Ising interface crossing in at least one of these annuli. In turn, the existence of more than $C \eta^{-2} (N-1) + 1$ disjoint Ising loops of diameter larger than $\varepsilon_1$ implies that one of the annulus in the cover contains at least $N$ disjoint Ising interface crossings.

As Ising interfaces form a subset of the FK dual configuration through the Edwards-Sokal coupling, FK dual paths provide similar crossings: for any integer $N$, for any scale $\eta<\varepsilon_1$, we can find a mesh size $\delta<\eta$ such that with probability at least $\varepsilon_2$, there exists an annulus of inner radius $\eta$ and outer radius $\varepsilon_1$ which is crossed by $N$ disjoint dual FK arms.

However \cite[Lemma 5.7]{chelkak-duminil-hongler} (via the use of quasi-multiplicativity \cite[Theorem 1.3]{chelkak-duminil-hongler} to compute arm exponents) implies that FK $N$-arm monochromatic (i.e., all arms are primal or all arms are dual) exponents are larger than $2$ for $N$ large enough (note that, as FK measures are positively correlated, these exponents also provide upper bounds on dual crossing of annuli that intersect the wired boundary of our domain $\Omega$). Hence we can find an integer $N$, such that, provided that the scale $\eta$ is small enough, the probability that at least one annuli among $C \eta^{-2}$ is crossed by $N$ disjoint dual FK arms is arbitrarily small, uniformly in the mesh size $\delta$. This yields a contradiction, and so the quantity $\underline{d}_n$ converges to $0$ as claimed.
\end{proof}

Let us now argue that the outermost Ising loops have a conformally
invariant scaling limit. Consider the critical Ising model with $+$
boundary conditions coupled with an FK model with wired boundary conditions
on a discretization $\left(\Omega_{\delta}\right)_{\delta>0}$ of
a Jordan domain $\Omega$. 
\begin{lem}
\label{lem:outermost-loops-have-limit} As the mesh size $\delta\to0$,
the leftmost level one Ising loops converge in law with respect to
the topology generated by $d_{\mathcal{X}}$ to a conformally invariant
scaling limit.

Furthermore, in the scaling limit, the Ising loops are contained in
the cut-out domains of the outermost FK loops. \end{lem}
\begin{proof}
We are going to describe the collection $\mathcal{L}$ of leftmost level one Ising loops iteratively as a countable union of loop collections $\mathcal{L}=\cup_n \mathcal{L}_n$. Convergence will follow from the a.s. convergence of each of the loop collections $\mathcal{L}_n$, as well as from the fact that the supremum of the diameter of the loops in $\mathcal{L}_n$ goes to $0$ in probability as $n\to\infty$, uniformly in $\delta$.

\begin{figure}
\centering \includegraphics[width=9cm]{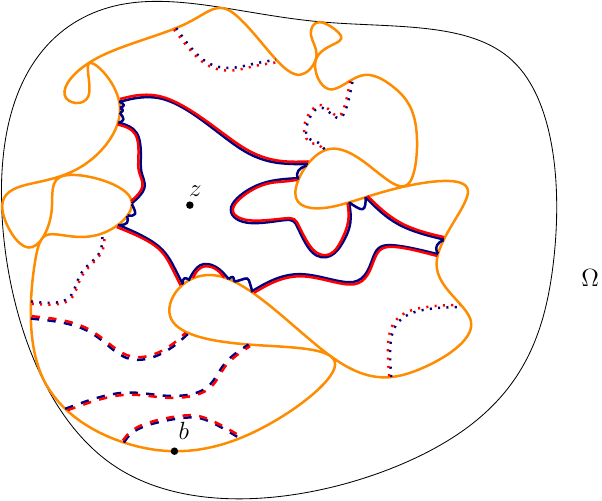}
\caption{In orange, an FK loop touching the boundary of the domain $\Omega$: the corresponding primal FK cluster carries $+$ spins. In red and blue, paths of Ising $+$ and $-$ spins respectively. The plain paths correspond to the collection of arcs $a(z,b)$, that together draw an Ising loop of level 1 surrounding $z$. In dashed, a few arcs belonging to $\mathcal{A}(z,b)$.}
\label{fig:FAE}
\end{figure}

Let us first describe $\mathcal{L}_1$ (see Figure \ref{fig:FAE}).  We start by conditioning on the outermost FK loops in $\Omega_{\delta}$. Let us denote by $\mathcal{C}_1$
the associated set of discrete cut-out domains in $\Omega_{\delta}$ (see Section \ref{sec:nested}). Any cut-out domain $C_{\delta}\in\mathcal{C}_1$ satisfies the following: it is bordered
on its outside by a strong path of $+$ and, conditionally on $C_{\delta}$,
the Ising spins inside of $C_{\delta}$ have free boundary conditions. Let us call $\mathcal{L}\left(C_{\delta}\right)$ the set of leftmost Ising loops in $C_{\delta}$ that touch the boundary of $C_{\delta}$ (note that we are in the setup of Lemma \ref{lem:arcsandloops}).

Now, consider the loop collection
\[
\mathcal{L}_1=\bigcup_{C_{\delta}\in\mathcal{C}_1} \mathcal{L}\left(C_{\delta}\right),
\]
where the union is over all cut-out domains $C_{\delta}\in\mathcal{C}_1$. We can pass the construction of $\mathcal{L}_1$ to the scaling limit. By Proposition \ref{prop:fk-loops-scaling-limit}, the discrete cut-out
domains $C_{\delta}\in\mathcal{C}_1$ converge to the continuous cut-out domains as $\delta\to 0$.
Moreover, for any cut-out domain $C=\lim C_{\delta}$, the Ising arc ensemble
in $C_{\delta}$ converges to the (conformally invariant) Free Arc
Ensemble in $C$ as $\delta\to0$ (Theorem \ref{thm:ising-main-theorem}). Hence, the boundary touching loops $\mathcal{L}\left(C_{\delta}\right)$ converge (via the correspondence between arcs and loops explained in Lemma \ref{lem:arcsandloops}). Consequently, the collection $\mathcal{L}_1$ has a conformally invariant scaling limit.

Recall now that any Ising loop is contained inside one of the cut-out domains $C_{\delta}\in\mathcal{C}_1$ (Remark \ref{rem:ising-loops-subset-dual-fk}). For such a cut-out domain $C_{\delta}$,
we have that the loops of $\mathcal{L}_1$ further
cut $C_{\delta}$ in regions of two types (this is the dichotomy of Lemma \ref{lem:arcsandloops}): 
\begin{itemize}
\item The regions enclosed by the loops of $\mathcal{L}_1$
(each loop in $\mathcal{L}_1$ separates
an inner weak circuit of $-$ on its inside and a strong circuit of
$+$ on its outside).
\item The regions that
are outside the loops of $\mathcal{L}_1$ (these
regions have strong $+$ boundary conditions). Let us denote by $\mathcal{R}\left(C_{\delta}\right)$ the set of these regions.
\end{itemize}
A loop $\ell_{\delta}$ that is inside of $C_{\delta}$ hence falls into one of three categories: 
\begin{itemize}
\item The loop $\ell_{\delta}$ is strictly contained inside of a loop of
$\mathcal{L}_1$: in this case, it is of level
two or higher (i.e. it is not outermost). 
\item The loop $\ell_{\delta}$ is contained inside of a loop $L_{\delta}\in\mathcal{L}_1$
and it shares an edge with $L_{\delta}$ (and hence is of level one,
but not leftmost if it is distinct from $L_{\delta}$). The structure
of these loops is described in Lemma \ref{lem:non-left-most-loop}
below. 
\item The loop $\ell_{\delta}$ is strictly contained in one of the regions in $\mathcal{R}\left(C_{\delta}\right)$. 
\end{itemize}
In order to find the remaining leftmost level 1 Ising loops (i.e., the collection $\mathcal{L}\setminus\mathcal{L}_1$) we hence just
need to look inside the regions in $\mathcal{R}\left(C_{\delta}\right)$.

Let us call 
\[
\mathcal{R}_2=\bigcup_{C_{\delta}\in\mathcal{C}_1} \mathcal{R}\left(C_{\delta}\right).  
\]
Any region $C_{\delta}\in\mathcal{R}_2$ carries strong $+$ boundary conditions, and its boundary is connected by a strong path of $+$ to the boundary of $\Omega_\delta$, so any leftmost level 1 Ising loop in $C_{\delta}$ is also a leftmost level 1 Ising loop of $\Omega_\delta$.

By resampling an FK representation with wired boundary conditions of the Ising model in each of the domains $C_{\delta}$, we can take the construction we just did for the unique region of $\mathcal{R}_1=\{\Omega_\delta\}$ (that yielded the loop collections $\mathcal{L}_1$) and apply this construction to each of the regions in $\mathcal{R}_2$. In particular, we construct a collection $\mathcal{C}_2$ of cut-out domains associated to the outermost FK loops in each region in $\mathcal{R}_2$, and use these to obtain a collection of leftmost level 1 Ising loops $\mathcal{L}_2$, as well as a set of smaller regions $\mathcal{R}_3$ with $+$ boundary conditions, that contain all the leftmost level 1 Ising loops that were neither in $\mathcal{L}_1$ nor in $\mathcal{L}_2$. We further iterate until all loops are found, so that $\mathcal{L}=\cup_n \mathcal{L}_n$. For each fixed $n$, the loop collection $\mathcal{L}_n$ has a conformally invariant scaling limit, for the same reason that $\mathcal{L}_1$ has.

To deduce the convergence of $\mathcal{L}=\cup_n \mathcal{L}_n$ from the convergence of the terms $\mathcal{L}_n$, we need them to uniformly converge in some sense: we will show that macroscopic Ising loops cannot belong to $\mathcal{L}_n$ for $n$ arbitrarily large. More precisely, we will now show that the quantity $\overline{d}_{\mathcal{L}_n}$ (the diameter of the largest loop in $\mathcal{L}_n$) tends to $0$ in probability as $n\to \infty$, uniformly in $\delta$. As the loops of $\mathcal{L}_n$ are contained in the domains belonging to $\mathcal{R}_n$, it is enough to show that the supremum $\overline{d}_{\mathcal{R}_n}$ of the diameters of the elements of $\mathcal{R}_n$ goes to $0$ in probability as $n\to \infty$, uniformly in $\delta$.

We now fix an integer $n\geq 1$, and consider an integer $1 \leq i \leq n+1$. We denote by $\mathcal{S}_i$ the set of collections of $n$ disjoint loops $\ell_1, \cdots ,\ell_n$ such that
\begin{itemize}
\item For $1 \leq j < i$, $\ell_j$ is the boundary of a domain in $\mathcal{R}_{j+1}$.
\item For $i \leq j \leq n$, $\ell_j$ is an Ising loop.
\item For any $1 \leq j < n$, $\ell_{j+1}$ is contained in the interior of $\ell_{j}$.
\end{itemize}

For $1 \leq i \leq n$, consider the following operation $F_i$ on spin configurations: condition first on the cut-out domains $\mathcal{C}_{i}$, and flip the Ising spins inside these domains. Conditioned on the cut-out domains $\mathcal{C}_{i}$, the Ising model inside of them has the law of independent Ising models with free boundary conditions. Hence the map $F_i$ is measure-preserving. Moreover the boundary of any domain in $\mathcal{R}_{i+1}$ becomes an Ising loop after the spin flip (by Lemma \ref{lem:arcsandloops}). Hence, any collection in $\mathcal{S}_{i+1}$ for the configuration $\sigma$ belongs to the set $\mathcal{S}_{i}$ for the configuration $F_i(\sigma)$. This gives the following estimate, for any $\varepsilon_1>0$:

\[
\mathbb{P}\left[ \sup_{\mathcal{C}\in\mathcal{S}_{i+1}} \underline{d}_{\mathcal{C}} \geq \varepsilon_1 \right] \leq \mathbb{P}\left[ \sup_{\mathcal{C}\in\mathcal{S}_{i}} \underline{d}_{\mathcal{C}} \geq \varepsilon_1 \right].
\]

If we piece together these estimates for $i$ going from $1$ to $n$, we see that

\[
 \mathbb{P}\left[\overline{d}_{\mathcal{R}_n} \geq \varepsilon_1 \right] \leq      \mathbb{P}\left[ \sup_{\mathcal{C}\in\mathcal{S}_{n+1}} \underline{d}_{\mathcal{C}} \geq \varepsilon_1 \right] \leq \mathbb{P}\left[ \sup_{\mathcal{C}\in\mathcal{S}_{1}} \underline{d}_{\mathcal{C}} \geq \varepsilon_1 \right]\leq \mathbb{P}\left[ \underline{d}_n \geq \varepsilon_1 \right],
\]
where  $\underline{d}_n$ is as in Lemma \ref{lem:finitelymanybigloops}. In other words, $\overline{d}_{\mathcal{R}_n}$ is stochastically dominated by $\underline{d}_n$. This yields the desired estimate on $\overline{d}_{\mathcal{R}_n}$ via Lemma \ref{lem:finitelymanybigloops}.
\end{proof}

\begin{lem}
\label{lem:non-left-most-loop}Any macroscopic level one Ising loop
is close to a leftmost level one loop, as in Theorem \ref{thm:main-thm}.
Moreover the set of rightmost level one loops has the same scaling
limit as the set of leftmost level one loops. \end{lem}
\begin{proof}
Any level one Ising loop $\ell$ is contained in a leftmost level one Ising loop $\ell^L$ (that we can for example construct as the boundary of the connected component containing $\ell$ of the complement of the set of vertices that are connected to the boundary by a strong path of $+$).
From the proof of Lemma \ref{lem:outermost-loops-have-limit}, we have that, uniformly in $\delta$, with high probability, there is a bounded number of leftmost level one Ising loop $\ell^L$ of diameter larger than a fixed $\varepsilon$, it is hence enough to prove the current lemma for a fixed leftmost Ising loop $\ell^L$.

Recall that such a loop $\ell^L$ can be recovered by gluing Ising arcs in a subdomain $R$ of $\Omega^\delta$ that carries free Ising boundary conditions. In particular (see Figure \ref{fig:lfrome}), if $a,b,c,d$ are four points on the boundary of $R$ that are in direct order, and such that $a$ and $b$ are not disconnected by $\ell^L$, and $c$ and $d$ are not disconnected by it either, but $a$ and $c$ are, then the loop $\ell^L$ can be obtained by gluing a subpath of the leftmost Ising exploration $\gamma^L_{a,b}$ from $a$ to $b$ (as defined in \cite{benoist-duminil-hongler}) with a subpath of the leftmost Ising exploration $\gamma^L_{c,d}$ from $c$ to $d$: indeed, the curve $\gamma^L_{a,b}$ contains all of $\ell^L$ except for the arc disconnecting $a$ from the interior of $\ell^L$.

\begin{figure}
\centering \includegraphics[width=12cm]{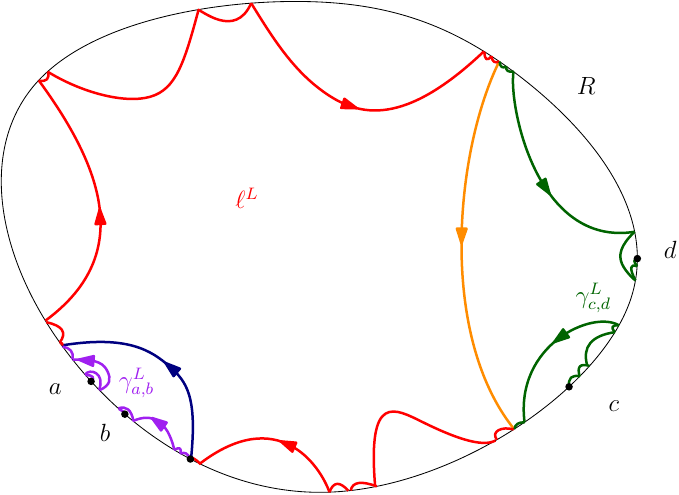}
\caption{In red, orange and blue, the loop $\ell^L$. In purple, red and orange, the exploration $\gamma^L_{a,b}$. In green, red and blue, the exploration $\gamma^L_{c,d}$.}
\label{fig:lfrome}
\end{figure}

In \cite[Section 4.2]{benoist-duminil-hongler} (see Lemma 18 and the argument following it), it is shown that the leftmost and rightmost Ising explorations from $a$ to $b$, $\gamma^L_{a,b}$ and $\gamma^R_{a,b}$ are close to each other in a strong sense: the supremum of the diameters of the connected components of
$\left(\gamma^L_{a,b}\cup\gamma^R_{a,b}\right)\setminus\left(\gamma^L_{a,b}\cap\gamma^R_{a,b}\right)$
goes to $0$ in probability as the mesh size $\delta$ goes to $0$ (this is a rephrasing of the last assertion of Theorem 6). This implies that, with high probability, we can glue a subpath of $\gamma^R_{a,b}$ together with a subpath of $\gamma^R_{c,d}$ and thus get a rightmost Ising loop $\ell^R$ that will be close to $\ell^L$ in the same strong sense. In particular, $\ell^R$ will share some edges with $\ell^L$ and hence be a level one loop.

Moreover, with high probability, any Ising loop $\ell$ contained in $\ell^L$ falls in one of three categories:
\begin{itemize}
\item One of the edges of $\ell$ belongs to $\ell^L\cap\ell^R$. In that case, we see that $\ell^L\cap\ell^R\subset\ell$ and $\ell$ is close to both $\ell^L$ and $\ell^R$ in the strong sense described above.
\item No edge of $\ell$ belongs to $\ell^L\cap\ell^R$, but $\ell$ shares an edge with $\ell^L$. In that case, the loop $\ell$ is contained in a connected components of
$\left(\ell^L\cup\ell^R\right)\setminus\left(\ell^L\cap\ell^R\right)$, hence is microscopic.
\item The loop $\ell$ shares no edge with $\ell^L$, in which case it is of level 2.
\end{itemize}

As a result, for any $ \varepsilon > 0 $, with probability tending to $ 1 $ as $ \delta \to 0 $, any loop of level 1 of diameter greater than $ \varepsilon $ contained in $ \ell^L $ is 
such that $ d_\Gamma (\ell, \ell^L) \leq \varepsilon $, and such that the connected components of $ ( \ell \cup \ell^L ) \setminus ( \ell \cap \ell^L ) $ have diameter less than $ \varepsilon $.
\end{proof}

\section{\label{sec:identification-scaling-limit}Identification of the Scaling
Limit}

In this section, we finish the proof of Theorem \ref{thm:main-thm}, by using the characterization of CLE.

\subsection{\label{sub:qualitative-scaling-limit}Qualitative Properties of the
Scaling Limit of Ising Loops}

\begin{lem}
\label{lem:limits-avoid-boundary} \label{lem:ising-loops-dont-touch-boundary}
In the scaling limit, the outermost Ising loops almost
surely do not touch the boundary. \end{lem}

In order to prove this, we will use the following notations. Consider the critical Ising model with $+$ boundary conditions on
the discretization $\Omega_{\delta}$ of a Jordan domain $\Omega$. Let us fix a real number $d\leq {\rm diam}(\Omega) /10$, and consider a cover of the boundary $\partial \Omega$ by a finite family of boundary subarcs $\mathcal{I}_1,\cdots,\mathcal{I}_m$ of diameter less than $d$.

For any boundary subarc $\mathcal{I}\subset\partial\Omega$, and any
positive real $r>0$, let $\mathcal{I}^{r}:=\left\{ z\in\Omega:\mathrm{d}\left(z,\mathcal{I}\right)\leq r\right\} $ be the $r$-neighborhood of $\mathcal{I}$.  For any positive reals
$\varepsilon>\eta>0$, let $E_{\delta}\left(\eta,\varepsilon,\mathcal{I}\right)$
be the event that there is a weak $-$ spin cluster linking $\mathcal{I}^{\eta}$
to $\Omega\setminus(\partial \Omega)^\varepsilon$ in $\mathcal{I}^{d}$.

\begin{lem}\label{lem:ising-loops-dont-touch-boundary-eq}
For any $\varepsilon>0$, we have 
\begin{equation}
\lim_{\eta\to0}\limsup_{\delta\to0}\mathbb{P}\left(\bigcup_j E_{\delta}\left(\eta,\varepsilon,\mathcal{I}_j\right)\right)=0.\label{eq:event}
\end{equation}
\end{lem}

\begin{proof}
Let us suppose by contradiction that this limit is positive: as we
will see, this will imply the occurrence of a zero-probability
event for SLE$_{3}$. If (\ref{eq:event}) were not to hold, we could
find some $\varepsilon>0$ and some boundary subarc $\mathcal{I}=\mathcal{I}_j$
such that 
\[
\limsup_{\eta\to0}\limsup_{\delta\to0}\mathbb{P}\left(E_{\delta}\left(\eta,\varepsilon,\mathcal{I}\right)\right)=\alpha>0.
\]
Let us consider a boundary subarc $\mathcal{J}$ that does not intersect $\mathcal{I}^{2d}$
and consider the critical Ising model on $\Omega_{\delta}$ with $+$
boundary conditions on $\partial\Omega\setminus\mathcal{J}$ and $-$
boundary conditions on $\mathcal{J}$ (denote by $\mathbb{P}^{+/-}$
the corresponding measure). By monotonicity with respect to the boundary
conditions, we would have 
\[
\limsup_{\eta\to0}\limsup_{\delta\to0}\mathbb{P}^{+/-}\left(E_{\delta}\left(\eta,\varepsilon,\mathcal{I}\right)\right)=\alpha'\geq\alpha>0.
\]

\begin{figure}
\centering \includegraphics[width=12cm]{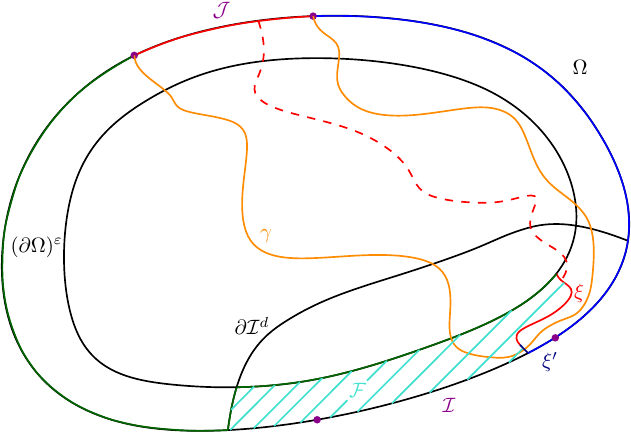}
\caption{In the domain $\Omega$, with boundary arcs $\mathcal{I}$ and $\mathcal{J}$ materialized by their extremities (purple points), we consider the Ising model with mixed boundary conditions $\mathbb{P}^{+/-}$, conditioned on $\xi$ to exist, and on $\mathcal{F}$, the data of the state of all spins in the dashed turquoise region.
We obtain an Ising model in a quad whose boundary consists of $\mathcal{J}$ (carrying $-$ spins, in red) ; the green path, carrying mixed spins ; the curve $\xi$ carrying $-$ spins, in red ; and a blue curve (in two shades), carrying mixed spins.
RSW estimates ensure the existence with uniformly positive probability of a strong path of $-$ spins (dashed red) in this quad. This path forces the Ising interface $\gamma$ generated by the mixed boundary conditions on $\Omega$ to get close to $\partial\Omega$ at the base of $\xi$.}
\label{fig:Ising_loops_boundary} 
\end{figure}

Let us fix sequences $\eta_{k}\to0$ as $k\to\infty$ and $\delta_{n,k}\to0$
as $n\to\infty$ such that for all $n,k$ we have 
\[
\mathbb{P}^{+/-}\left(E_{\delta_{n,k}}\left(\eta_{k},\varepsilon,\mathcal{I}\right)\right)\geq\frac{\alpha'}{2}>0.
\]
Let us consider the first crossing $\xi$ of $-$ spins in $\mathcal{I}^{d}$ from $\mathcal{I}^{\eta_{k}}$
to $\Omega\setminus (\partial \Omega)^\varepsilon$ in counterclockwise
order. We arbitrarily extend $\xi$ to $\partial \Omega$ by adding to it a path $\xi'$ staying in $\mathcal{I}^{\eta_k}$

We condition on the data $\mathcal{F}$ of all the Ising spins included in  $\mathcal{I}^{d}$ to the left of $\xi\cup\xi'$.
Conditionally on $\xi$ to exist, and on the data $\mathcal{F}$, the law of the remaining spins is given by an Ising model in a quad, where two of the opposing boundary subintervals ($\xi$ and $\mathcal{J}$) carry $-$ spins (see Figure \ref{fig:Ising_loops_boundary}). Moreover, the extremal length of this quad is bounded, uniformly in $\xi$, $\delta$.

By RSW-type estimates \cite[Corollary 1.7]{chelkak-duminil-hongler},
we obtain that with (uniformly) positive probability, $\xi$ is connected
by a path of $-$ spins to the $-$ spins of $\mathcal{J}$. This
implies that the Ising interface generated by the $\pm$ boundary
condition hits $\mathcal{I}^{\eta_{k}}$. Passing to the limit $n\to\infty$
and then $k\to\infty$ implies that the scaling limit of the interface
generated by the boundary conditions hits $\mathcal{I}$ with positive
probability, which is a contradiction: the scaling limit is $\mathrm{SLE}_{3}$
\cite[Theorem 1]{chelkak-duminil-hongler-kemppainen-smirnov}, which
never hits the boundary of the domain it lives in \cite[Proposition 6.8]{lawler-ii}. \end{proof}

\begin{proof}[Proof of Lemma \ref{lem:ising-loops-dont-touch-boundary}]
Lemma \ref{lem:ising-loops-dont-touch-boundary-eq} shows that the scaling limits of Ising loops almost surely either do not touch the boundary or are included in it. However, this second possibility almost surely does not happen, as Ising loops are included in cut-out domains of the scaling limits of the FK loops, themselves described by a branching $\mathrm{SLE}_{16/3}\left(-2/3\right)$. On the event that an Ising loop is included in the boundary, the same would have to be true of a piece of $\mathrm{SLE}_{16/3}\left(-2/3\right)$, which cannot happen.

The interested reader can find an inspiration for an argument that does not use the FK coupling but RSW-type considerations instead in \cite[Lemma 18]{benoist-duminil-hongler}
\end{proof}

\begin{lem}
\label{lem:loops-limits-are-simple}The scaling limit of the outermost
Ising loops are almost surely simple and do not touch each other.\end{lem}
\begin{proof}
As explained in the proof of Lemma \ref{lem:outermost-loops-have-limit}, we can discover
the collection of outermost Ising loops $\mathcal{L}_1$ through the iteration of
an exploration process of FK loops and Ising arcs. At the $n$-th, step, we discover a sub-collection $\mathcal{L}_n$

Two Ising loops with different discovery times $m<n$ do not touch each other in the scaling limit. Indeed,
the loops in $\cup_{n>m} \mathcal{L}_n$ (i.e., the loops that are discovered at times strictly larger than $m$) are outermost loops within one of the cut-out domains of the collection $\mathcal{R}_m$. By Lemma \ref{lem:limits-avoid-boundary}, the loops in $\cup_{n>m} \mathcal{L}_n$ stay strictly inside the domains of $\mathcal{R}_n$ (with the notation of the proof of Lemma \ref{lem:outermost-loops-have-limit}) and hence cannot intersect the loops in $\mathcal{L}_m$ 

We now show that the loops of a sub-collection $\mathcal{L}_m$ are almost
surely simple and do not touch each other. Recall that these Ising loops are recovered
by gluing arcs of a FAE staying in a cut-out domain of some FK loops (in the collection $\mathcal{R}_m$).
By Proposition \ref{prop:cut-out_domains}, the boundaries of cut-out
domains in $\mathcal{R}_m$ are simple disjoint curves. Moreover, the Ising loops that
stay within a single cut-out domain are simple and disjoint: indeed,
these loops are constructed by gluing arcs of the FAE, which is a
family of simple disjoint arcs. \end{proof}
\begin{lem}
\label{lem:loops-domain-markov-property}In the scaling limit, the
outermost Ising loops satisfy the domain Markov property. \end{lem}
\begin{proof}
This directly follows from the discrete Markov property, together
with the convergence of outermost Ising loops for any discrete approximation
of the continuous domain (Lemma \ref{lem:outermost-loops-have-limit}). 
\end{proof}

Finally, we will need to identify the value of the parameter $\kappa=3$.
\begin{lem}
\label{lem:loops-local-structures}Consider the scaling limit of the
outermost Ising loops. Almost surely, there is a subarc of a loop that
has Hausdorff dimension 11/8. \end{lem}
\begin{proof}
As explained in the proof of Lemma \ref{lem:outermost-loops-have-limit},
the Ising loops can be constructed by gluing arcs of the FAE, which
are pieces of an $\mathrm{SLE}_{3}\left(-\frac{3}{2},-\frac{3}{2}\right)$
exploration tree. As a result, Ising loops have subarcs that have Hausdorff dimension 11/8 (see \cite{beffara}). 
\end{proof}

\subsection{\label{sub:proof-main-thm}Proof of Theorem \ref{thm:main-thm}}
\begin{proof}
The outermost Ising loops have a conformally invariant scaling limit
(Lemma \ref{lem:outermost-loops-have-limit}). By Lemmas \ref{lem:limits-avoid-boundary}
and \ref{lem:loops-limits-are-simple}, the loops of the scaling limit
are simple and do not touch each other or the boundary. By Lemma \ref{lem:loops-domain-markov-property},
these loops satisfy the Markovian restriction property. Hence, by
the Sheffield-Werner Markovian characterization property, the scaling
limit of the outermost Ising loops is a $\mathrm{CLE}_{\kappa}$ for
some $\frac{8}{3}<\kappa\leq4$. By Lemma \ref{lem:loops-local-structures}
and the construction of $\mathrm{CLE}_{\kappa}$ in terms of $\mathrm{SLE}_{\kappa}$-like
loops \cite{sheffield}, we deduce that $\kappa$ must be equal to
$3$.

We now establish the convergence of the Ising loops of level $2$.
These loops lay inside of the outermost Ising loops, i.e. are loops
of an Ising configuration with $-$ boundary conditions. The set of
leftmost loops in a domain with $-$ boundary conditions has the same
law as the set of rightmost loops in a domain with $+$ boundary conditions.
By the second part of Lemma \ref{lem:non-left-most-loop}, the scaling
limit of level 2 loops is hence given by independent $\mathrm{CLE}_{3}$
in each of the level 1 loops. Further iterating this argument, we
identify the set of loops of level $n$ as independent $\mathrm{CLE}_{3}$
inside the loops of level $n-1$.

Moreover, a corollary of Lemma \ref{lem:finitelymanybigloops} is that the supremum of the diameters of the Ising loops of level $n$ goes to $0$ in probability as $n\to\infty$, uniformly in the mesh size $\delta$. Hence the convergence of the set of all Ising loops to nested $\mathrm{CLE}_{3}$ follows from the convergence of the Ising loops of level $n$, for each fixed $n$. 
\end{proof}

\appendix

\section{FK interfaces have a conformally invariant scaling limit}

In this section, we prove Proposition \ref{prop:fk-loops-scaling-limit}:
we show that for any Jordan domain $\Omega$, the set of all level
one (outermost) FK loops in $\Omega$ converges towards a conformally
invariant limit (see also \cite{kemppainen-smirnov-ii} for an alternative
approach). We prove the convergence of a single FK exploration path
(Lemma \ref{lem:cvFKinterface}) without smoothness assumptions on
$\partial\Omega$.

We rely on \cite{kemppainen-smirnov} to get tightness of the exploration
path, and we rely on the convergence result \cite[Theorem 2]{chelkak-duminil-hongler-kemppainen-smirnov}
together with an argument similar to the one used in \cite{benoist-duminil-hongler}
to identify the scaling limit. Once we know the convergence of a single
exploration, we can deduce the convergence of all outermost FK loops
by iteratively using the convergence of FK exploration interfaces
(Proposition \ref{prop:cvallFK}).

Let us also point out that \cite{kemppainen-smirnov-iii} have recently given an alternative proof of the results stated here, with a more in-depth description of the relevant loops.

\subsection{Convergence of FK interfaces}

Consider a discrete domain $\Omega_{\delta}$, with dual $\Omega_{\delta}^{*}$
and bi-medial $\Omega_{\delta}^{b}$. Consider an FK model configuration
$\omega$ on $\Omega_{\delta}$ with wired boundary conditions and
its dual configuration $\omega^{*}$. Given boundary bi-medial points
$a$ and $b$, we define the \emph{FK exploration $\gamma_{\delta}$
of $\omega$ from $a$ to $b$} as follows: 
\begin{itemize}
\item $\gamma_{\delta}$ is a simple path on the bi-medial graph of $\Omega_{\delta}$
from $a$ to $b$ 
\item when $\gamma_{\delta}$ is not on the boundary $\partial\Omega_{\delta}^{b}$,
$\gamma_{\delta}$ follows the primal boundary cluster of $\omega$
on its left and a dual cluster of $\omega^{*}$ on its right 
\item when $\gamma_{\delta}$ is on $\partial\Omega_{\delta}^{b}$, $\gamma_{\delta}$
keeps a dual cluster on its right whenever possible (i.e. with the
constraint that $\gamma_{\delta}$ is simple and goes from $a$ to
$b$). \end{itemize}
The following result gives us that the interface $ \gamma_\delta $ converges to 
the ${\rm SLE}_{\kappa}(-\rho)$ process with $ \kappa = 16/3 $ and $ \rho = \kappa - 6 $
(see e.g. \cite{schramm-wilson, sheffield}).
\begin{lem}
\label{lem:cvFKinterface} Let $\Omega_{\delta}$ be a discrete approximation
of a Jordan domain $\Omega$ and let $a_{\delta},b_{\delta}\in\partial\Omega_{\delta}^{b}$
with $a_{\delta},b_{\delta}\to a,b\in\partial\Omega$. Let $\gamma_{\delta}$
be the FK exploration from $a_{\delta}$ to $b_{\delta}$. Then $\gamma_{\delta}$
converges in law to an ${\rm SLE}_{16/3}(-2/3)$ curve, with respect
to the supremum norm up to reparametrization. \end{lem}
\begin{proof}
For any discrete time step $n\geq0$, let us denote by $\Omega_{n,\delta}$
the connected component of $\Omega_{\delta}\setminus\gamma_{\delta}\left[0,n\right]$
containing $b$. Let us denote by $O_{n,\delta}$ the clockwise-most
point of $\partial\Omega_{n,\delta}\cap\partial\Omega_{\delta}$ (i.e.
the rightmost intersection of $\gamma_{\delta}\left[0,n\right]$ with
$\partial\Omega_{\delta}$). Conditionally on $\gamma_{\delta}\left[0,n\right]$,
the configuration $\omega$ in $\Omega_{n,\delta}$ is an FK configuration
with boundary conditions that are free on the right of $\gamma_{\delta}$
(i.e. on the counterclockwise arc $\left[\gamma_{\delta}\left(n\right),O_{n,\delta}\right]$)
and wired elsewhere.

Analogously to the reasoning made in \cite{benoist-duminil-hongler},
we can use the crossing estimates of \cite{chelkak-duminil-hongler}
on the domains $\left(\Omega_{n,\delta}\right)$ to apply the technology
of \cite[Theorems 3.4 and 3.12]{kemppainen-smirnov} which allows
us to obtain that the law of the exploration $\gamma_{\delta}$ is
tight.

Let us now consider an almost sure scaling limit $\gamma$ of $\gamma_{\delta_{k}}$
for $\delta_{k}\to0$ (which is possible via the Skorokhod embedding
theorem). Consider a conformal mapping $\Omega\to\mathbb{H}$ with
$(a,b)\mapsto(0,\infty)$ and denote by $\lambda$ the image of $\gamma$.
Encode $\lambda$ by a Loewner chain, i.e. consider the family of
conformal mappings $g_{t}:H_{t}\to\mathbb{H}$, where $H_{t}$ is
the unbounded connected component of $\mathbb{H}\setminus\lambda\left[0,t\right]$
and we normalize $t$ and $g_{t}$ such that, for any time $t$, $g_{t}\left(z\right)=z+2t/z+o\left(1/z\right)$
as $z\to\infty$. Set $U_{t}:=g_{t}\left(\lambda\left(t\right)\right)$
and $O_{t}:=\sup\left(\lambda\left[0,t\right]\cap\mathbb{R}\right)$.
As usual, we have the Loewner equation $\partial_{t}g_{t}\left(z\right)=2/\left(g_{t}\left(z\right)-U_{t}\right)$
and as a result $U_{t}$ characterizes $\lambda$
and hence $\gamma$.

Let us now characterize the law of $U_{t}$. Set
\[
X_{t}:=\frac{O_{t}-U_{t}}{\sqrt{16/3}}.
\]
The following
two properties allow one to prove that $X_{t}$ is
a Bessel process of dimension $d=3/2$ (as in \cite{benoist-duminil-hongler}): 
\begin{itemize}
\item The process $X_{t}$ is instantaneously reflecting
off $0$, i.e the set of times $\left\{ t:X_{t}=0\right\} $ has zero
Lebesgue measure: this is a deterministic property for Loewner chains
(see e.g. \cite[Lemma 2.5]{miller-sheffield}). This imply that $X_{t}$
can be deterministically recovered from the ordered set of its excursions
away from $0$. 
\item When $X_{t}$ is away from $0$, the tip of the curve
$\gamma\left(t\right)$ is away from the boundary arc $\left[a,b\right]$.
This corresponds in the discrete to a time $t$ when the curve $\gamma_{\delta_{k}}\left(t\right)$
is away from the boundary arc $\left[a_{\delta_{k}},b_{\delta_{k}}\right]$,
and hence the domain yet to be explored $\Omega_{t,\delta_{k}}$ has
`macroscopic Dobrushin conditions' in the sense that both the wired
and free boundary arcs are macroscopic. By \cite[Theorem 2]{chelkak-duminil-hongler-kemppainen-smirnov},
the curve $\lambda$ behaves as chordal $\mathrm{SLE}\left(16/3\right)$
headed towards $O_{t}$ until the first hitting time of $[O_{t},+\infty)$,
which is the same, by SLE coordinate change \cite{schramm-wilson},
as an SLE$_{16/3}(-2/3)$ with force point at $O_{t}$ until the first
hitting time of $[O_{t},+\infty)$. Hence, the law of the ordered
set of excursions of $X_{t}$ away from $0$ is that
of the ordered set of excursions of a Bessel process of dimension
$d=3/2$ away from $0$. 
\end{itemize}
Let us now identify the process $O_{t}$. By integrating the Loewner
equation, we find that
\[
O_{t}=\int_0^t\frac{2}{\sqrt{16/3}\;X_{s}}ds+A_{t},
\]
where $A_{t}$ is constant when $X_{t}$ is away
from $0$. The following argument yields that $A_{t}$
is constant equal to zero: 
\begin{itemize}
\item The process $X_{t}$ is $\left(\frac{1}{2}\right)^{-}$-H\"older
continuous, as a Bessel process. 
\item Since $d=3/2>1$, we have that
\[
I_{t}=\int_{0}^{t}\frac{2}{\sqrt{16/3}\;X_{s}}ds
\]
is almost surely finite, and moreover $I_{t}$ is
$\left(\frac{1}{2}\right)^{-}$-H\"older continuous (as can be seen
from the fact that $X_{t}-I_{t}$ is a standard Brownian motion). 
\item By the tightness result of \cite{kemppainen-smirnov} (see \cite[Theorems 3.4 and 3.12]{kemppainen-smirnov}),
the driving function $U_{t}$ must be a.s. $\left(\frac{1}{2}\right)^{-}$-H\"older
continuous. 
\item As $A_{t}=\sqrt{16/3}\;X_{t}+U_{t}-I_{t}$, we see that $A_{t}$
must be $\left(\frac{1}{2}\right)^{-}$-H\"older continuous. 
\item Since $A_{t}$ can only vary on the set $\left\{ t:X_{t}=0\right\} $,
whose dimension is $\frac{2-d}{2}=\frac{1}{4}<\frac{1}{2}$, $A_{t}$
must be constant (as in \cite[Lemma 9]{benoist-duminil-hongler} or
\cite[Section 5.4]{kemppainen-smirnov-ii}). 
\end{itemize}
This characterizes the law of the pair $\left(X_{t},O_{t}\right)$,
and hence the law of the driving function $U_{t}$:
the curve $\lambda$ is an SLE$_{16/3}(-2/3)$ process, with force
point starting at $0^{+}$. 
\end{proof}

\subsection{Special points of FK interfaces}

The goal of Section \ref{sec:cfFKloops} will be to describe a conformally
invariant exploration procedure which for any $\varepsilon>0$, with
high probability as $\delta\to0$, gives $\varepsilon$-approximations
of all the level one FK loops of diameter larger than $\varepsilon$.

In order to do so, we need to control the formation of continuous
cut-out domains, i.e. we need to understand what happens on the lattice
level when the scaling limit of the interface has a double point or
touches the boundary. We only need to control what happens on the
right side of the curve $\gamma_{\delta}$, as the arguments needed
to control the left side are similar. Let us now introduce the following
subsets of $\gamma_{\delta}$ (resp. $\gamma$): 
\begin{itemize}
\item The set $\mathcal{P}_{B}$ of \emph{right boundary points} is the
set of points $x$ where the interface $\gamma_{\delta}$ comes as
close as possible to the counter-clockwise boundary arc from $a$
to $b$, i.e. a bi-medial mesh size $\delta/2$ away (resp. $\gamma$
touches the boundary arc $\left[a,b\right]$ in the continuous). 
\item The set $\mathcal{P}_{D}$ of \emph{clockwise double points} is the
set of points $x$ where $\gamma_{\delta}$ comes within distance
$\delta/2$ of a point $x'$ of its past (resp. $x$ is a double point
of $\gamma$), and such that the interface winds clockwise from $x'$
to $x$. 
\end{itemize}
We call points of $\mathcal{P}_{B}\cup\mathcal{P}_{D}$ \emph{special
points}. For a special point $x\in\mathcal{P}_{B}\cup\mathcal{P}_{D}$
of $\gamma_{\delta}$ (resp. $\gamma$), we define the subpath $K\left(x\right)$
of $\gamma_{\delta}$ (resp. of $\gamma$) as follows: 
\begin{itemize}
\item If $x\in\mathcal{P}_{B}$, $K(x)$ is the whole part of the curve
running from the origin $a$ to $x$ (excluding $a$ and $x$). 
\item If $x\in\mathcal{P}_{D}$, $K(x)$ is the loop the curve forms at
$x$ (excluding $x$), i.e. the part of the curve $\gamma_{\delta}$ (resp. $\gamma$)
running from $x'$ to $x$, where $x'$ is as above. Note that this is well-defined in the continuum, as the FK interface does not have triple points: this would indeed imply a six-arm event primal-dual-primal-dual-primal-dual that is ruled out by \cite[Remark 1.6]{chelkak-duminil-hongler}.
\end{itemize}
Note that for any special points $x,y\in\mathcal{P}_{B}\cup\mathcal{P}_{D}$,
if $K(x)\cap K(y)\neq\emptyset$, then either $K(x)\subseteq K(y)$
or $K(y)\subseteq K(x)$.

Given a finite family of points $\mathcal{P}\subset\mathcal{P}_{B}\cup\mathcal{P}_{D}$
containing the endpoint $b$, we define the partition $(P(x))_{x\in\mathcal{P}}$
of $\gamma_{\delta}\setminus\{a,b\}$ (resp. $\gamma\setminus\{a,b\}$) by: 
\[
P(x):=K(x)\setminus\bigcup_{y\in\mathcal{P}:K(y)\subset K(x)}K(y).
\]
Note that $P(x)$ is always a non-self-crossing loop (minus a point) or a non-self-crossing path joining two boundary points.
For $x\in\mathcal{P}$, we say that $P\left(x\right)$ is \emph{an
$\varepsilon$-approximation of a cut-out domain} if (at least) one of
the following holds: 
\begin{itemize}
\item The set $P(x)$ is of diameter less than $2\varepsilon$.
\item We can write the set $P(x)$ as
\[
P(x)=\left((\partial D \cap \Omega) \setminus {x}\right) \cup \bigcup_{{\gamma'} \in \mathcal{C}} {\gamma'},
\]
 where $D$ is a cut-out domain laying on the right side of $\gamma_{\delta}$ (resp. $\gamma$), and $\mathcal{C}$ is a collection of paths of diameter less than $\varepsilon$, with one endpoint on $\partial D \setminus {x}$
\end{itemize}
We say that a finite family of special points $\mathcal{P}\subset\mathcal{P}_{B}\cup\mathcal{P}_{D}$
is an \emph{$\varepsilon$-pinching family} if for each point $x\in\mathcal{P}$,
$P\left(x\right)$ is an $\varepsilon$-approximation of a cut-out domain.

The main statement of this subsection is that double points of the
scaling limit correspond to double points of the discrete interface
(and similarly for the boundary points). More precisely, we have the
following: 
\begin{prop}
\label{prop:fktouches} For all $\varepsilon,\varepsilon'>0$, there is
a mesh size $\delta_{0}>0$, such that for all $\delta<\delta_{0}$,
the following holds:

We can couple $\gamma_{\delta}$ and $\gamma$, find an $\varepsilon$-pinching
family $\mathcal{P}^{\varepsilon}$ of $\gamma$ and parametrize $\gamma_{\delta}:\left[0,1\right]\to\Omega_{\delta}$
and $\gamma:\left[0,1\right]\to\bar{\Omega}$ so that, with probability
at least $1-\varepsilon'$: 
\begin{enumerate}
\item For any special point $x\in\mathcal{P}^{\varepsilon}$, there is a special
point $x_{\delta}$ of $\gamma_{\delta}$ such that the subpaths $K(x_{\delta})$
and $K(x)$ are parametrized by the same time intervals. 
\item The parametrized curves $\gamma_{\delta}$ and $\gamma$ are $\varepsilon$-close
to each other in the topology of supremum norm. 
\end{enumerate}
\end{prop}
\begin{proof}
By contradiction, assume there is $\varepsilon,\varepsilon'$ and a sequence of mesh size $\delta_n \to 0$ such that we cannot do the required construction for any of the values $\delta_n$. By Skorokhod embedding theorem, and the convergence in law of $\gamma_{\delta}\to\gamma$, we can construct a coupling of
$\gamma_{\delta_{n}}$ and $\gamma$ such that $\gamma_{\delta_{n}}\to\gamma$
almost surely. Let us construct discrete $\varepsilon/2$-pinching families
$\mathcal{P}_{\delta_n}^{\varepsilon}$, and
show that we can extract a subsequence $\left(\delta_{k}\right)\subset\left(\delta_{n}\right)$
such that $(\gamma_{\delta_{k}},\mathcal{P}_{\delta_{k}}^{\varepsilon})$
converges almost surely to $(\gamma,\mathcal{P}^{\varepsilon})$, where
$\mathcal{P}^{\varepsilon}$ is a finite $\varepsilon$-pinching family (and convergence of pinching points is in the sense of (1) in the statement of the theorem).
This will provide a contradiction, and hence imply the claim.

\begin{figure}
\centering \includegraphics[width=12cm]{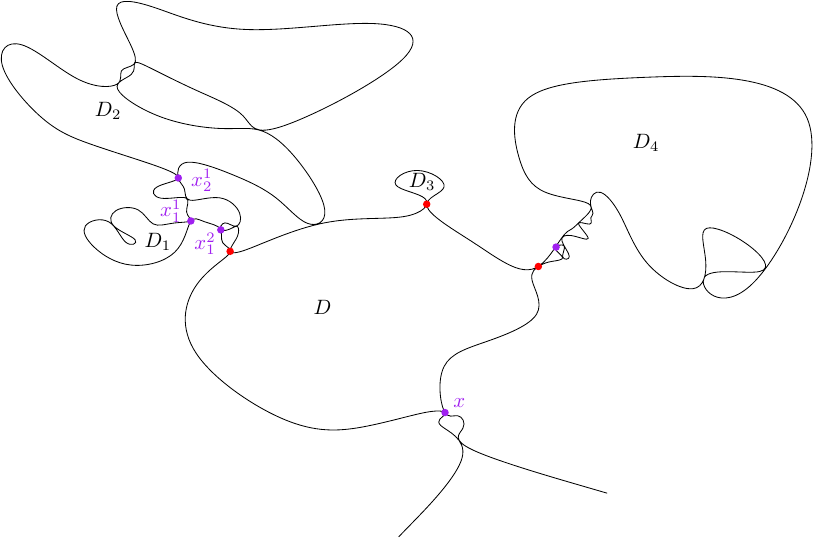}
\caption{A cut-out domain $D$ of diameter larger than $\varepsilon$ (the scale is given in orange). It is bounded
by pieces of $\gamma$, and clockwise double points, shown in red.
Some of these double points disconnect pieces of diameter less than
$2\varepsilon$ (the domain $D_{3}$ is one of these). The other (macroscopic)
double points are controlled on the discrete level by purple pinching
points (disconnecting $D_{1}$, $D_{2}$ and $D_{4}$). Note that
in this picture the point $x_{2}^{2}$ is equal to $x_{1}^{2}$.}
\label{fig:FKpinching} 
\end{figure}

Let us now consider the finite family $\left(w_{k}\left(\delta\right)=\gamma_{\delta}\left(s_{k}\right)\right)_{k=0,\ldots,n}$
of points on $\gamma_{\delta}$ in chronological order, with 
\begin{itemize}
\item $w_{0}\left(\delta\right)=a$ and $w_{n}\left(\delta\right)=b$ 
\item for any $0<i<n$, $s_{i}$ is the first time after $s_{i-1}$ when
$\gamma_{\delta}$ exits the ball of radius $\varepsilon/4$ around $w_{i-1}\left(\delta\right)$, 
\item $\mathrm{diam}\left(\gamma_{\delta}\left[s_{n-1},s_{n}\right]\right)\leq\varepsilon/4$. 
\end{itemize}
We then define a finite $\varepsilon/2$-pinching family $\mathcal{P}_{\delta}^{\varepsilon}$
of points $x_{i}^{j}\left(\delta\right)=\gamma_{\delta}\left(t_{i}^{j}\left(\delta\right)\right)$
(that may not be distinct) in the following way: 
\begin{itemize}
\item Let $x_{i}^{0}\left(\delta\right)$ be the first point of $\mathcal{P}_{B}\cup\mathcal{P}_{D}$
coming after $w_{i}\left(\delta\right)$ such that $K\left(x_{i}^{0}\left(\delta\right)\right)$
contains $w_{i}\left(\delta\right)$ and is of diameter at least $\varepsilon/2$. 
\item For $j\geq0$, we define $x_{i}^{j+1}\left(\delta\right)$ as the
first point after $x_{i}^{j}\left(\delta\right)$ such that $K(x_{i}^{j}\left(\delta\right))\subset K(x_{i}^{j+1}\left(\delta\right))$
and such that the diameter of $K(x_{i}^{j+1}\left(\delta\right))\setminus K(x_{i}^{j}\left(\delta\right))$
is at least $\varepsilon/2$. 
\end{itemize}
Let us explain why $\mathcal{P}_{\delta}^{\varepsilon}$ is an $\varepsilon/2$-pinching
family (see Figure \ref{fig:FKpinching}). 
\begin{itemize}
\item Suppose $y$ is a special point of $\gamma_{\delta}$ such that ${\rm diam}K(y)\geq\varepsilon/2$.
Then $K(y)$ contains at least one point $w_{i}(\delta)$. Let $j$
be the highest index such that $x_{i}^{j}(\delta)\in K(y)$. Then
we have that either $y=x_{i}^{j}(\delta)$, or that $K(y)\setminus K(x_{i}^{j}(\delta))$
is of diameter less than $\varepsilon/2$. In particular, for any special
point $y$, we have that either $y\in\mathcal{P}_{\delta}^{\varepsilon}$
or that the diameter of 
\[
K(y)\setminus\bigcup_{x_{i}^{j}\left(\delta\right)\in\mathcal{P}_{\delta}^{\varepsilon}:K(x_{i}^{j}\left(\delta\right))\subset K(y)}K(x_{i}^{j}\left(\delta\right))
\]
is less than $\varepsilon/2$. 
\item Every special point $x$ of $\gamma_{\delta}$ is the closing point
of a (possibly degenerate, i.e. of diameter $\delta/2$) cut-out domain
$D_{x}$. Moreover, the set 
\[
K(x)\setminus\bigcup_{y\in\mathcal{P}_{B}\cup\mathcal{P}_{D}:K(y)\subset K(x)}K(y)
\]
corresponds to the parts of $\gamma_{\delta}$ that trace the boundary
of the cut-out domain $D_{x}$. 
\item By the two previous items, for any $x_{i}^{j}\left(\delta\right)\in\mathcal{P}_{\delta}^{\varepsilon}$,
we have that 
\[
P(x_{i}^{j}\left(\delta\right))=D_{x_{i}^{j}(\delta)}\bigcup\left(\bigcup_{y\in\partial D_{x_{i}^{j}(\delta)}\cap(\mathcal{P}_{B}\cup\mathcal{P}_{D}\setminus\mathcal{P}_{\delta}^{\varepsilon})}\left(K(y)\setminus\bigcup_{x_{k}^{l}(\delta)\in\mathcal{P}_{\delta}^{\varepsilon}\cap K(y)}K(x_{k}^{l}(\delta))\right)\right)
\]
can be obtained by attaching to $D_{x_{i}^{j}\left(\delta\right)}$
paths of diameter less than $\varepsilon/2$ and hence is an $\varepsilon/2$-approximation
of a cut-out domain. 
\end{itemize}
Note that the number of points in $\mathcal{P}_{\delta}^{\varepsilon}$
is tight, as the almost sure limit $\gamma$ of the $\gamma_{\delta_{n}}$
is a continuous curve. By compactness, we can assume that the family
$\mathcal{P}_{\delta}^{\varepsilon}$ converges to a finite family $\mathcal{P}^{\varepsilon}$
of points $x_{i}^{j}=\gamma(t_{i}^{j})$ on the curve $\gamma$ (at
least by taking a subsequence $\gamma_{\delta_{k}}$) in the following
sense: we can parametrize $\gamma_{\delta_{k}}$ and $\gamma$ such
that for these parametrizations $\|\gamma_{\delta_{k}}-\gamma\|_{\infty}\to0$
almost surely, such that the limit $t_{i}^{j}:=\lim t_{i}^{j}\left(\delta_{k}\right)$
exists for all $i,j$, and such that for any $x\in\mathcal{P}^{\varepsilon}$,
there exists a point $x_{\delta}\in\mathcal{P}_{\delta}^{\varepsilon}$
such that the subpaths $K(x_{\delta})$ and $K(x)$ are parametrized
by the same time intervals. The claim of the proposition hence reduces
to proving that the family $\mathcal{P}^{\varepsilon}$ is an $\varepsilon$-pinching
family for $\gamma$, namely that for any $x_{i}^{j}=\gamma\left(t_{i}^{j}\right)\in\mathcal{P}^{\varepsilon}$,
the set $P\left(x_{i}^{j}\right)$ is an $\varepsilon$-approximation
of a cut-out domain. In the following, we will assume that $x_{i}^{j}\in\mathcal{P}_{D}$.
The case $x_{i}^{j}\in\mathcal{P}_{B}$ follows from similar arguments
and is simpler.

To prove this, let us first introduce approximations $x_{i}^{j}\left(\delta,\eta\right)=\gamma_{\delta}\left(t_{i}^{j}\left(\delta,\eta\right)\right)$
of the $\varepsilon$-pinching points $x_{i}^{j}(\delta)$ for $\eta>0$.
The point $x_{i}^{j}\left(\delta,\eta\right)$ is defined to be the
first point after $x_{i}^{j-1}\left(\delta\right)$ (or after $w_{i}\left(\delta\right)$
if $j=0$) to be\emph{ $\eta$-close to being a macroscopic pinching
point}, i.e. such that if $\gamma_{\delta}$ were to continue along
a (hypothetical) path $\tilde{\gamma}$ of length at most $\eta$,
the pinching point $x_{i}^{j}\left(\delta\right)$ would be located
at the end of $\tilde{\gamma}$.

Consider the family of points $y_{i}^{j}$ located at times $u_{i}^{j}:=\lim_{\eta_{m}\to0}\lim_{\delta_{k}\to0}t_{i}^{j}\left(\delta_{k},\eta_{m}\right)$
(where the double limit $\delta_{k}\ll\eta_{m}\ll1$ is taken through
a diagonal extraction if necessary). It remains to show that with
high probability, each point $y_{i}^{j}=\gamma\left(u_{i}^{j}\right)$
appears chronologically right before the point $x_{i}^{j}$, in the
sense that $\mathrm{diam}\left(K(x_{i}^{j})\setminus K(y_{i}^{j})\right)\leq\varepsilon/2$.
This will readily imply that $P\left(x_{i}^{j}\right)$ is an $\varepsilon$-approximation
of a cut-out domain, thus concluding the proof of the proposition.

\begin{figure}
\centering \includegraphics[width=12cm]{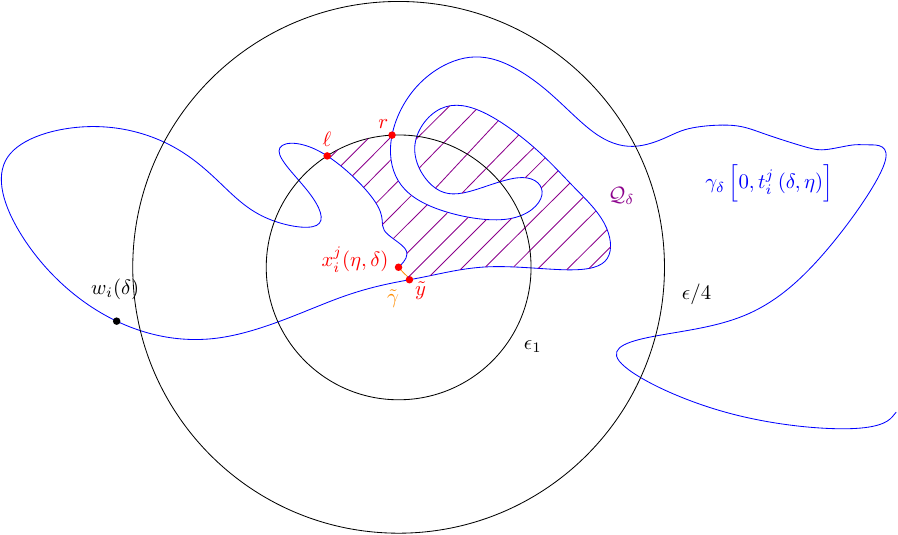} \caption{The quad $\mathcal{Q}_{\delta}$ in striped purple. Primal and dual
FK crossings in this quad will force the appearance of the pinching
point $x_{i}^{j}(\delta)$ before $\gamma_{\delta}$ touches the arc
$[r,\ell]$.}
\label{fig:quad} 
\end{figure}

Let us now fix a point $y_{i}^{j}$. We want to find a quad (i.e.
a topological rectangle) $\mathcal{Q}_{\delta}$ that contains (with
high probability) FK crossings ensuring that the set $K(x_{i}^{j})\setminus K(y_{i}^{j})$
is small. For $\alpha>0$, let us define the ball $B_{\alpha}:=\left\{ z\in\Omega_{\delta}:\left|z-x_{i}^{j}\left(\delta,\eta\right)\right|<\alpha\right\} $
of radius $\alpha$ around $x_{i}^{j}\left(\delta,\eta\right)$, and
let $T_{\alpha}$ be the set of times when $\gamma_{\delta}$ visits
this ball: 
\[
T_{\alpha}\left(\delta,\eta\right):=\left\{ t\in\left[0,t_{i}^{j}\left(\delta,\eta\right)\right]:\gamma_{\delta}(t)\in B_{\alpha}\right\} .
\]
The curve $\gamma$ does not have triple points, as this would produce
a six-arm event prevented by \cite[Remark 1.6]{chelkak-duminil-hongler}.
In particular $y_{i}^{j}$ is not a triple point for $\gamma$. As
a result, we can pick $\varepsilon_{1}>0$ small enough such that for
all $\delta\ll\eta\ll\varepsilon_{1}$, we have that $T_{\varepsilon_{1}}\left(\delta,\eta\right)$
is included in two connected components of $T_{\varepsilon/4}$ with
high probability: one of the connected components containing $t_{i}^{j}\left(\delta,\eta\right)$,
and the other one containing the end $\tilde{y}$ of the (hypothetical)
curve $\tilde{\gamma}$ considered above (see Figure \ref{fig:quad}).

Let us now define the quad $\mathcal{Q}_{\delta}$, with boundary
marked points $r,\ell,x_{i}^{j}\left(\delta,\eta\right)$ and $\tilde{y}$
(in counterclockwise order): 
\begin{itemize}
\item the segment $\left[r,\ell\right]$ is the connected component of $\partial B_{\varepsilon_{1}}\setminus\gamma_{\delta}\left[0,t_{i}^{j}\left(\delta,\eta\right)\right]\cap\partial B_{\varepsilon_{1}}$
that disconnects $x_{i}^{j}\left(\delta,\eta\right)$ from the endpoint
$b$ of $\gamma_{\delta}$ in the domain $\Omega_{\delta}\setminus\gamma_{\delta}\left[0,t_{i}^{j}\left(\delta,\eta\right)\right]$; 
\item the segments $\left[\ell,x_{i}^{j}\left(\delta,\eta\right)\right]$
and $\left[\tilde{y},r\right]$ follow $\gamma_{\delta}$; 
\item the segment $\left[x_{i}^{j}\left(\delta,\eta\right),\tilde{y}\right]$
is simply $\tilde{\gamma}$. 
\end{itemize}
By choosing $\eta$ small enough (and $\delta\ll\eta$), we can make
the extremal distance (see \cite{chelkak-duminil-hongler} for a definition)
between the arcs $\left[\ell,x_{i}^{j}\left(\delta,\eta\right)\right]$
and $\left[\tilde{y},r\right]$ in $\mathcal{Q}_{\delta}$ arbitrarily
small. By the RSW estimate of \cite{chelkak-duminil-hongler}, we
can ensure that with arbitrarily high probability (for any $\delta$
small enough), there is a dual FK crossing separating $\left[x_{i}^{j}\left(\delta,\eta\right),\tilde{y}\right]$
from $\left[r,\ell\right]$ and furthermore a primal FK crossing separating
the dual crossing from $\left[r,\ell\right]$. As a result, there
is a point $z_{\ell}\in\left[\ell,x_{i}^{j}\left(\delta,\eta\right)\right]$
and a point $z_{r}\in\left[\tilde{y},r\right]$ between which $\gamma_{\delta}$
has to travel while staying within $\mathcal{Q}_{\delta}$. In particular,
$z_{r}$ is a special point of $\gamma_{\delta}$ such that $K(x_{i}^{j-1}(\delta))\subset K(z_{r})$
(resp. such that $w_{i}(\delta)\in K(z_{r})$ if $j=0$) and such
that $K(z_{r})\setminus K(x_{i}^{j-1}(\delta))$ encloses $\tilde{\gamma}$
and hence is of diameter larger than $\varepsilon/2$. This ensures that
the point $x_{i}^{j}(\delta)$ is found before the point $z_{r}$,
in particular before $\gamma_{\delta}$ crosses the arc $[r,\ell]$.
As $y_{i}^{j}$ is not a triple point for the curve $\gamma$, and
as the quad $\mathcal{Q}_{\delta}$ is of diameter less than $\varepsilon/2$,
we see that with high probability, $\gamma\left[t_{i}^{j}(\eta),t_{i}^{j}\right]$
is of diameter less than $\varepsilon/2$ and hence with high probability
the set $K(x_{i}^{j})\setminus K(y_{i}^{j})$ is of diameter less
than $\varepsilon/2$.

By taking the successive limits $\delta\ll\eta\ll\varepsilon_{1}\ll\varepsilon$,
we obtain that all $P\left(x_{i}^{j}\right)$ are $\varepsilon$-approximations
of a cut-out domain, and hence that $\mathcal{P}^{\varepsilon}$ is an
$\varepsilon$-pinching family, thus proving the proposition. 
\end{proof}

\subsection{Convergence of FK loops}

\label{sec:cfFKloops}

In this subsection, we identify the scaling limit of the outermost
FK loops by a recursive exploration procedure (see Figure \ref{fig:FKloops}),
analogous to the exploration procedure introduced in \cite{camia-newman-ii}. 
\begin{prop}
\label{prop:cvallFK} The set of all level 1 FK loops in the discrete
approximation $\Omega_{\delta}$ of a Jordan domain $\Omega$ with wired boundary conditions converges
to a conformally invariant scaling limit. \end{prop}
\begin{proof}
Let us fix $\varepsilon>0$ and define an exploration procedure (see
\cite{camia-newman-ii} for a similar construction) that will discover
all the level one FK loops of diameter at least $\varepsilon$. Note that, in order to see that the $d_\mathcal{X}$ distance is smaller than $\varepsilon$, only the loops of diameter at least $\varepsilon$ matter.

\begin{figure}
\centering \includegraphics[width=12cm]{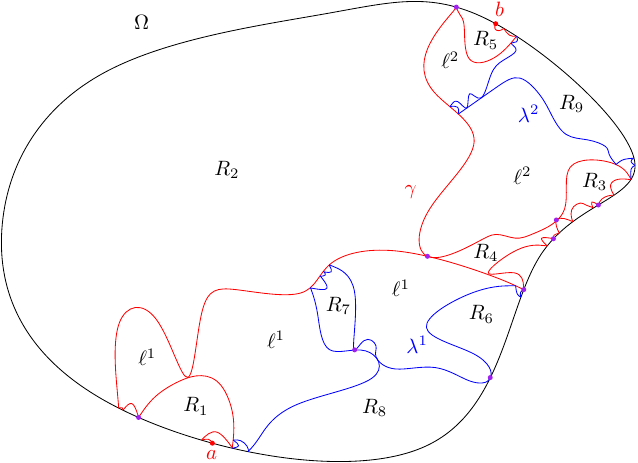} \caption{In red, the first exploration path. It cuts out $k=2$ domains with
mixed boundary conditions and large diameter, as well as five domains
$R_{1},\ldots,R_{5}$ with wired boundary conditions and large diameter.
The mixed boundary conditions domain are further cut by two blue explorations
$\lambda^{1}$ and $\lambda^{2}$. These cut out four more large domains
with wired boundary conditions, and allows us to recover two FK loops:
the three domains $R_{6},R_{7},R_{8}$ and the loop $\ell_{1}$ for
$\lambda^{1}$; and the domain $R_{9}$ and the loop $\ell_{2}$ for
$\lambda^{2}$. The purple dots are special discrete points with very
high probability (as they are $\varepsilon$-pinching points).}
\label{fig:FKloops} 
\end{figure}

Let us choose boundary bi-medial points $a_{\delta},b_{\delta}\in\partial\Omega_{\delta}^{b}$
converging to points $a,b\in\partial\Omega$, and consider the exploration
path $\gamma_{\delta}$ from $a_{\delta}$ to $b_{\delta}$, as in
Lemma \ref{lem:cvFKinterface}. Let us condition on $\gamma_{\delta}$
and consider the connected components of $\Omega_{\delta}\setminus\gamma_{\delta}$: 
\begin{itemize}
\item The connected components on the left side of $\gamma_{\delta}$ have
wired boundary conditions. 
\item The components on the right side of $\gamma_{\delta}$ that touch
the boundary of $\Omega_{\delta}$ have mixed boundary conditions. 
\item All other components stay on the right side of $\gamma_{\delta}$
and have free boundary conditions. 
\end{itemize}
Let us consider the \emph{macroscopic domains cut-out by} $\gamma_{\delta}$,
i.e., the connected components of $\Omega_{\delta}\setminus\gamma_{\delta}$
of diameter at least $\varepsilon$. As $\delta\to0$, the interface $\gamma_{\delta}$ converges to a continuous curve, namely an $\mathrm{SLE}_{16/3}\left(-2/3\right)$, (Lemma \ref{lem:cvFKinterface}) and as all double points of this limit correspond to discrete double points of $\gamma_{\delta}$ (Proposition \ref{prop:fktouches}), we see that, uniformly in the mesh size $\delta$, with high probability there are at most $N$ macroscopic domains cut-out by $\gamma_{\delta}$, for $N$ large enough.

In each of these macroscopic
domains $D_{\delta}^{j}$ (for $j=1,\ldots,k\leq N$) that have mixed
boundary conditions, we consider the FK interface $\lambda_{\delta}^{j}$
that separates the wired and the free boundary arcs. We obtain $k$
FK loops $\ell_{\delta}^{j}$ by concatenating each interface $\lambda_{\delta}^{j}$
with the arc of $\gamma_{\delta}$ joining its endpoints. Moreover,
each of the $\lambda_{\delta}^{j}$ cuts the mixed domain into a collection
of domains with free boundary conditions (these are the cut-out domains
of $\ell_{\delta}^{j}$) and domains with wired boundary conditions.

The interface $\gamma_{\delta}$ converges to $\mathrm{SLE}_{16/3}\left(-2/3\right)$
as $\delta\to0$ (Lemma \ref{lem:cvFKinterface}). As we control special
points of $\gamma_{\delta}$ (Proposition \ref{prop:fktouches}), the
domains $D_{\delta}^{j}$ converge to continuous connected components
$D^{j}$ of $\Omega\gamma$ (in the sense that their boundaries converge as curves for the supremum norm up to reparametrization). Furthermore, for each $j$, the interface
$\lambda_{\delta}^{j}$ converges to an $\mathrm{SLE}_{16/3}$ curve
in $D^{j}$ as $\delta\to0$ (\cite[Theorem 2]{chelkak-duminil-hongler-kemppainen-smirnov}).
With high probability, the complement of the interior of the loop
$\ell^{j}$ in the domain $D^{j}$ has less than $M$ connected component
(carrying wired boundary conditions) of diameter larger than $\varepsilon$.

We have hence explored a batch of (at most $N$) level $1$ FK loops
and with high probability (with $N$ fixed and $\delta\to0$), the
region outside of these loops contains at most $N+NM$ wired domains
of diameter larger than $\varepsilon$.

\begin{figure}
\includegraphics{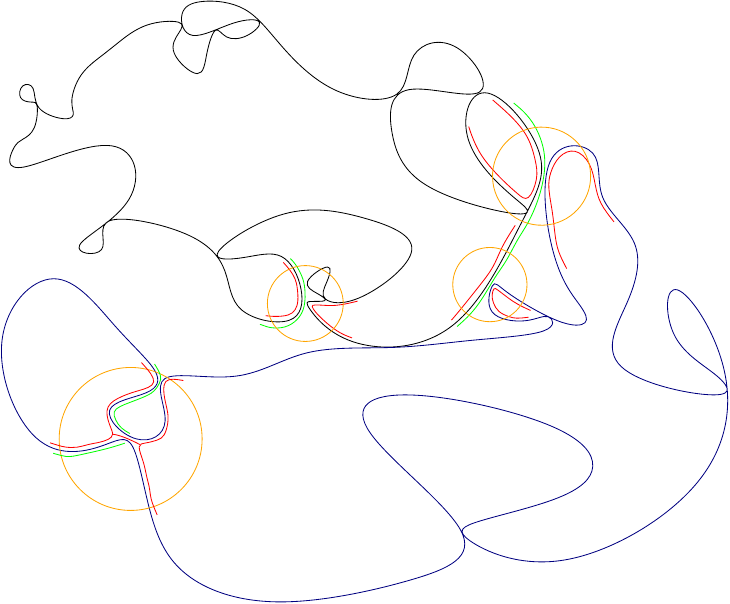} \caption{\label{fig:fk-six-arm}Double-points and contact points of FK cut-out
domains cannot happen in the scaling limit, as they correspond to
six-arm events. Two FK loops, in black and blue. Part of the primal
configuration is drawn in green, part of the dual configuration in
red. Six-arm events are marked by orange circles at four different
locations. From left to right: cases (B), (A), (C), (C) of the proof
of Proposition \ref{prop:cut-out_domains}.}
\end{figure}

Each of these domains can be further explored by iterating the exploration
we used for $\Omega$: starting with an interface between two far
away points on the boundary of these new domains, and starting secondary
explorations in all the resulting mixed domains of diameter larger
than $\varepsilon$.

Each step of the exploration scheme reduces the maximum diameter of
domains in the collection of wired domains yet to be explored (which
are connected components of the set of points laying outside all FK
loops currently discovered), and so, by choosing a number of step
$n$ large enough, we can ensure, with arbitrarily high probability,
that after $n$ iterations of this scheme, we are left with only domains
of diameter less than $\varepsilon$.

Note that when this is the case, any level 1 FK loop that has not
been found needs to stay in one of the small wired domains cut out,
and so is of diameter less than $\varepsilon$. \end{proof}
\begin{rem}
\label{rem:fktouches} Note that the argument of Proposition \ref{prop:fktouches}
tells us that all double points, contact points or boundary touching
points of the scaling limits of FK loops are limits of discrete double
points, contact points, and boundary touching points. 
\end{rem}

\subsection{The boundary of cut-out domains are disjoint simple curves}

We conclude this appendix by a qualitative property of continuous
FK loops.
\begin{prop}
\label{prop:cut-out_domains} The boundary of continuous cut-out domains
are disjoint simple curves. \end{prop}
\begin{proof}
By construction, the cut-out domains do not have `internal' double
points, i.e. double points that would disconnect their interior. Now,
the presence of an `external' double point (i.e. a point that would
disconnect the interior of their complement) would imply the presence
of a six-arm event (dual-dual-primal-dual-dual-primal, in cyclic order)
for the FK model as in Figure \ref{fig:fk-six-arm}, case (A). In
the scaling limit, this is ruled out by \cite[Theorem 1.5]{chelkak-duminil-hongler}
(using the same argument as in Remark 1.6 there). Moreover, the boundaries
of macroscopic cut-out domains do not touch each other by a similar
argument. If there were a point of intersection, this would again
imply a six-arm event (dual-dual-primal-dual-dual-primal, in cyclic
order), which is again ruled out: see Figure \ref{fig:fk-six-arm}
for the two sub-cases (B) and (C) of this case. \end{proof}

\end{document}